\documentclass[a4paper]{amsart}

\usepackage{hyperref,geometry,graphicx,amsmath,amssymb,amsfonts,amsthm,array,latexsym}
\usepackage[utf8]{inputenc}

\theoremstyle{plain}
\newtheorem{thm}{Theorem}[section]
\newtheorem{cor}[thm]{Corollary}
\newtheorem{lem}[thm]{Lemma}

\newtheorem{prop}[thm]{Proposition}

\theoremstyle{definition}
\newtheorem{defn}[thm]{Definition}
\newtheorem{rem}[thm]{Remark}
\newtheorem{exmp}[thm]{Example}

\def\leq{\leqslant}
\def\geq{\geqslant}

\def\ZZ{\mathbb{Z}}

\def\PP{\mathbb{P}}

\def\Res{\mathrm{Res}}
\def\Disc{\mathrm{Disc}}

\def\UU{\mathbb{U}}
\def\Mon{\mathrm{Mon}}
\def\Dod{\mathrm{Dod}}
\def\Mt{\mathbb{M}}
\def\Dt{\mathbb{D}}

\title[Resultant of a  $\mathfrak{S}_{n}$-equivariant polynomial system]{Resultant of an equivariant polynomial system with respect to the symmetric group}

\author{Laurent\ Bus\'e} 
\address{Email: laurent.buse@inria.fr, 
INRIA Sophia Antipolis-M\'editeran\'ee, France.} 

\author{Anna\ Karasoulou} 
\address{Email: akarasou@di.uoa.gr, 
Department of Informatics \& Telecommunications,
National and Kapodistrian University of Athens, Greece.}

\begin{document}

\begin{abstract}
	Given a system of $n\geq 2$ homogeneous polynomials in $n$ variables which is equivariant with respect to the canonical actions of the symmetric group of $n$ symbols on the variables and on the polynomials, it is proved that its resultant can be decomposed into a product of several smaller resultants that are given in terms of some divided differences. As an application, we obtain a decomposition formula for the discriminant of a multivariate homogeneous symmetric polynomial. 
\end{abstract}

\maketitle

\section{Introduction}

The analysis and solving of polynomial systems are fundamental problems in computational algebra. In many applications, polynomial systems are highly structured and it is very useful to develop specific methods in order to  take into account a particular structure. In this paper, we will focus  on systems of $n$ homogeneous polynomials $f_{1},\ldots,f_{n}$ in $n$ variables $x_{1},\ldots,x_{n}$ that are globally invariant under the action of the symmetric group $\mathfrak{S}_{n}$ of $n$ symbols. More precisely, we will assume that  for any integer $i \in \{1,2,\ldots,n\}$ and any permutation $\sigma \in \mathfrak{S}_{n}$
\begin{equation*}
\sigma(f_{i}):=f_{i}(x_{\sigma(1)},x_{\sigma(2)},\ldots ,x_{\sigma(n)})=f_{\sigma(i)}(x_{1},x_{2},\ldots,x_{n}).
\end{equation*}
In the language of invariant theory these systems are called equivariant with respect to the symmetric group $\mathfrak{S}_{n}$, or simply $\mathfrak{S}_{n}$-equivariant (see for instance  \cite[\S 4]{Wor94} or \cite[Chapter 1]{DiCa71}). Some recent interesting developments based on Gr\"obner basis techniques for this kind of systems can be found in \cite{FS12} with applications. In this work, we will study the resultant of these systems. 

\medskip

The  main result of this paper (Theorem \ref{thm:maintheorem}) is a decomposition of the resultant of a $\mathfrak{S}_{n}$-equivariant polynomial system. This formula allows to split such a resultant into several other resultants that are in principle easier to compute and that are expressed in terms of the divided differences of the input polynomial system.  We emphasize that the multiplicity of each factor appearing in this decomposition is also given. Another important point of our result is that it is an exact and universal formula which is valid over the universal ring of coefficients (over the integers) of the input polynomial system. Indeed, we payed attention to use a correct and universal definition of the resultant. In this way, the formula we obtain has the correct geometric meaning and stays valid over any coefficient ring by specialization. This kind of property is particularly important for applications in the fields of number theory and arithmetic geometry where the value of the resultant is as important as its vanishing.

\medskip

The discriminant of a homogeneous polynomial is also a fundamental tool in computational algebra. Although the discriminant of the generic homogeneous polynomial of a given degree is irreducible, for a particular class of polynomials it can be decomposed and this decomposition is always deeply connected to the geometric properties of this class of polynomials. The second main contribution of this paper is a decomposition of the discriminant of a homogeneous symmetric polynomial (Theorem \ref{thm:discmaintheorem}). This result was actually the first goal of this work that has been inspired by the unpublished (as far as we know) note \cite{PeSh09} by N.~Perminov and S.~Shakirov where a first tentative for such a formula is given without a complete proof. Another motivation was also to improve the computations of discriminants for applications in convex geometry, following a paper by J.~Nie where the boundary of the cone of non-negative polynomials on an algebraic variety is studied by means of discriminants \cite{Nie12}.  We emphasize that our formula is obtained as a byproduct of our first formula on the resultant of a $\mathfrak{S}_{n}$-equivariant polynomial system. Therefore, it inherits from the same features, namely it allows to split a discriminant into several resultants that are easier to compute and it is a universal formula where the multiplicities of the factors are provided. Here again, we payed attention to use a correct and universal definition of the discriminant. 

\medskip

The paper is organized as follows. In Section \ref{sec:prem} we first provide some preliminaries on some material that we will need, namely multivariate divided differences, resultants and discriminants. Section \ref{sec:resultant} will be devoted to the main result of this paper (Theorem \ref{thm:maintheorem}), that is to say a decomposition formula for the resultant of a polynomial system which is $\mathfrak{S}_{n}$-equivariant. As a corollary of this formula, a decomposition of the discriminant of a homogeneous symmetric polynomial (Theorem \ref{thm:discmaintheorem}) is provided in Section \ref{sec:discriminant}. 

\section{Preliminaries}\label{sec:prem}

In this section we introduce our notation and the material we will use, namely divided differences, resultants and discriminants. We will provide proofs concerning the results on divided differences because we were not able to find the properties we needed  in the literature, although these results are part of the folklore and are definitely known to the experts.  

\subsection{Divided differences}\label{subsec:divdiff}

Let $R$ be a commutative ring and denote by $R[x_{1},\ldots,x_{n}]$ the ring of polynomials in $n\geq 2$ variables which is graded with the usual weights: $\deg(x_{i})=1$ for all $i\in \{1,\ldots,n\}$.
For any sequence of integers $1\leq i_{1}<i_{2}< \cdots < i_{k}\leq n$ we will denote by $V(x_{i_{1}},x_{i_{2}},\ldots,x_{i_{k}})$  the  Vandermonde determinant
$$V(x_{i_{1}},x_{i_{2}},\ldots,x_{i_{k}}):=\prod_{1\leq s<r \leq k} (x_{i_{r}}-x_{i_{s}}) = \det
\left|
\begin{array}{cccc}
 1 & x_{i_{1}} & \cdots & x_{i_{1}}^{{k-1}} \\
 1 & x_{i_{2}} & \cdots & x_{i_{2}}^{{k-1}} \\
 \vdots & \vdots & & \vdots \\
  1 & x_{i_{k}} & \cdots & x_{i_{k}}^{{k-1}} \\
\end{array}
\right|.
$$
It is a homogeneous polynomial in $R[x_{1},\ldots,x_{n}]$ of degree $\binom{k}{2}=\frac{k(k-1)}{2}$. For the sake of simplicity in the notation, for any integer $p$ the set $\{1,2,\ldots,p\}$ will be denoted by $[p]$ and given a finite set $I$, $|I|$ will stand for its cardinality.

\medskip

Suppose given $n$ homogeneous polynomials $P^{ \{1\} }, P^{\{2\}}, \ldots, P^{\{n\}}$ of the same degree $d\geq 1$ in $R[x_{1},\ldots,x_{n}]$ such that for all couple of integers $(i,j) \in [n]^{2}$ the polynomial $P^{\{i\}}-P^{\{j\}}$ is divisible by $x_{i}-x_{j}$:
\begin{equation}\label{eq:PiPj}
P^{\{i\}}-P^{\{j\}} \in (x_{i}-x_{j}) \subset R[x_{1},\ldots,x_{n}].
\end{equation}

\begin{lem}\label{lem:divdiff} For any set of $ k \geq 2$  distinct integers $\{ i_{1},\ldots,i_{k}\} \subset [n]$, there exists a unique homogeneous polynomial $P^{\{ i_{1},\ldots,i_{k}\}}$ in $R[x_{1},\ldots,x_{n}]$ of degree $d-k+1$ such that 
\begin{equation*}
V(x_{i_{1}},x_{i_{2}},\ldots,x_{i_{k}})\cdot P^{\{ i_{1},\ldots,i_{k}\}}(x_{1},\ldots,x_{n})=
\det
\left|
\begin{array}{ccccc}
 1 & x_{i_{1}} & \cdots & x_{i_{1}}^{{k-2}} & P^{\{i_{1}\}}(x_{1},\ldots,x_{n}) \\
 1 & x_{i_{2}} & \cdots & x_{i_{2}}^{{k-2}} & P^{\{i_{2}\}}(x_{1},\ldots,x_{n})\\
 \vdots & \vdots & & \vdots & \vdots \\
  1 & x_{i_{k}} & \cdots & x_{i_{k}}^{{k-2}} & P^{\{i_{k}\}}(x_{1},\ldots,x_{n}) \\
\end{array}
\right|. 
\end{equation*}
\end{lem}
\begin{proof} From the assumption \eqref{eq:PiPj} it is clear that $(x_{i}-x_{j})$ divides the Vandermonde-like determinant
 $$
\left|
\begin{array}{ccccc}
 1 & x_{i_{1}} & \cdots & x_{i_{1}}^{{k-2}} & P^{\{i_{1}\}} \\
 1 & x_{i_{2}} & \cdots & x_{i_{2}}^{{k-2}} & P^{\{i_{2}\}}\\
 \vdots & \vdots & & \vdots & \vdots \\
  1 & x_{i_{k}} & \cdots & x_{i_{k}}^{{k-2}} & P^{\{i_{k}\}} \\
\end{array}
\right|$$
and hence that $V(x_{i_{1}},x_{i_{2}},\ldots,x_{i_{k}})$ also divides it.
Now, $R$ being again an arbitrary commutative ring, the uniqueness of $P^{\{ i_{1},\ldots,i_{k}\}}$ follows from the fact that $V(x_{i_{1}},x_{i_{2}},\ldots,x_{i_{k}})$ is not a zero divisor in $R[x_{1},\ldots,x_{n}]$, which is a consequence of the Dedekind-Mertens Lemma (see for instance \cite[\S 2.4]{BuJo12}). 
\end{proof}

\begin{defn}\label{defdd} For all positive integer $k\leq n$, the polynomials $P^{\{i_{1},\ldots,i_{k}\}}$ defined in Lemma \ref{lem:divdiff} are called $(k-1)^{\textrm{th}}$ divided differences of the polynomials $P^{\{1\}},\ldots,P^{\{n\}}$. We notice that $P^{\{ i_{1},\ldots,i_{k}\}}=0$ if $d+1< k \leq n$. 
 \end{defn}

The first divided differences $P^{\{i,j\}}$ are easily seen to satisfy the equality
\begin{equation}\label{eq:1stdivdiff}
 (x_{i}-x_{j})P^{\{i,j\}}=P^{\{i\}}-P^{\{j\}}.
\end{equation}
This explains the terminology ``divided difference''. It turns out that similar equalities hold for the higher order divided differences.

\begin{prop}\label{prop:divdiff} 
Let $\{ i_{1},\ldots,i_{k} \}$ be a subset of $[n]$ with $k \geq 2$. Then, for any two distinct integers $p,q$ in $\{i_{1},\ldots,i_{k}\}$,
$$(x_{i_{q}}-x_{i_{p}})P^{\{ i_{1},i_{2},\ldots,i_{k} \}}=P^{\{i_{1},i_{2},\ldots,i_{k} \} \setminus \{i_{p}\}  }- P^{\{i_{1},i_{2},\ldots,i_{k} \} \setminus \{i_{q}\}}.$$
\end{prop}
\begin{proof} We observe that it is enough to prove this result over the universal ring of coefficients of $P^{ \{1\} }, P^{\{2\}}, \ldots, P^{\{n\}}$ over the integers and we proceed by induction on $k$. As we already noticed in \eqref{eq:1stdivdiff}, the claimed formula holds for $k=2$. So, we fix an integer $k>2$ and we assume that the claimed formula holds for any set $\{i_{1},\ldots,i_{r}\}$ of cardinality $\leq k-1$.  Observe also that since $P^{ \{ i_{1},i_{2},\ldots,i_{k} \}}$ is independent of the order of $i_{1},i_{2},\ldots,i_{k}$,  it is sufficient to prove the claimed equality for $\{p,q\}=\{1,2\}$.  

By definition (see Lemma \ref{lem:divdiff}), we have
$$
V(x_{i_{1}},\ldots,x_{i_{k}}) P^{\{i_{1},\ldots,i_{k} \}} =
 \left|
\begin{array}{ccccc}
 1 & x_{i_{1}} & \cdots & x_{i_{1}}^{{k-2}} & P^{\{i_{1}\}} \\
 1 & x_{i_{2}} & \cdots & x_{i_{2}}^{{k-2}} & P^{\{i_{2}\}}\\
 \vdots & \vdots & & \vdots & \vdots \\
  1 & x_{i_{k}} & \cdots & x_{i_{k}}^{{k-2}} & P^{\{i_{k}\}} \\
\end{array}
\right|.
$$
We denote by $\Delta$ this determinant. 
By subtracting the last row from all the other rows in the matrix of $\Delta$, we get
$$
\Delta=
 \left|
\begin{array}{ccccc}
 0 & x_{i_{1}} - x_{i_{k}} & \cdots & x_{i_{1}}^{{k-2}} -x_{i_{k}}^{{k-2}} & P^{\{i_{1}\}} - P^{\{i_{k}\}} \\
 0 & x_{i_{2}} - x_{i_{k}} & \cdots & x_{i_{2}}^{{k-2}} -x_{i_{k}}^{{k-2}} & P^{\{i_{2}\}} - P^{\{i_{k}\}} \\
 \vdots & \vdots & & \vdots & \vdots \\
  0 & x_{i_{k-1}} - x_{i_{k}} & \cdots & x_{i_{k-1}}^{k-2}-x_{i_{k}}^{{k-2}} & P^{\{i_{k-1}\}} - P^{\{i_{k}\}} \\
  1 & x_{i_{k}} & \cdots & x_{i_{k}}^{{k-2}} & P^{\{i_{k}\}} \\
\end{array}
\right| =  \left(  \prod_{j=1}^{k-1} (x_{i_{j}}-x_{i_{k}}) \right) \tilde{\Delta}
$$
where (we use \eqref{eq:1stdivdiff})
\begin{align*}
  \tilde{\Delta} &=
\left|
\begin{array}{ccccccc}
 0 & 1 & x_{i_{1}} + x_{i_{k}} &  \sum_{r=0}^{2} x_{i_{1}}^{r}x_{i_{k}}^{2-r} &  \cdots & \sum_{r=0}^{k-3} x_{i_{1}}^{r}x_{i_{k}}^{k-3-r} & P^{\{i_{1},i_{k}\}}  \\
 0 & 1 & x_{i_{2}} + x_{i_{k}} &  \sum_{r=0}^{2} x_{i_{2}}^{r}x_{i_{k}}^{2-r} &  \cdots & \sum_{r=0}^{k-3} x_{i_{2}}^{r}x_{i_{k}}^{k-3-r} & P^{\{i_{2},i_{k}\}} \\
 \vdots &  \vdots & \vdots & & \vdots & \vdots \\
 0 & 1 & x_{i_{k-1}} + x_{i_{k}} &  \sum_{r=0}^{2} x_{i_{k-1}}^{r}x_{i_{k}}^{2-r} &  \cdots & \sum_{r=0}^{k-3} x_{i_{k-1}}^{r}x_{i_{k}}^{k-3-r} & P^{\{i_{k-1},i_{k}\}} \\
  1 & x_{i_{k}} & x_{i_{k}}^{2} & x_{i_{k}}^{3} & \cdots & x_{i_{k}}^{{k-2}} & P^{\{i_{k}\}} \\
\end{array}
\right| \\
&= (-1)^{k-1}
\left|
\begin{array}{cccccc}
 1 & x_{i_{1}} + x_{i_{k}} &  \sum_{r=0}^{2} x_{i_{1}}^{r}x_{i_{k}}^{2-r} &  \cdots & \sum_{r=0}^{k-3} x_{i_{1}}^{r}x_{i_{k}}^{k-3-r} & P^{\{i_{1},i_{k}\}}  \\
 1 & x_{i_{2}} + x_{i_{k}} &  \sum_{r=0}^{2} x_{i_{2}}^{r}x_{i_{k}}^{2-r} &  \cdots & \sum_{r=0}^{k-3} x_{i_{2}}^{r}x_{i_{k}}^{k-3-r} & P^{\{i_{2},i_{k}\}} \\
  \vdots & \vdots & & \vdots & \vdots \\
  1 & x_{i_{k-1}} + x_{i_{k}} &  \sum_{r=0}^{2} x_{i_{k-1}}^{r}x_{i_{k}}^{2-r} &  \cdots & \sum_{r=0}^{k-3} x_{i_{k-1}}^{r}x_{i_{k}}^{k-3-r} & P^{\{i_{k-1},i_{k}\}} \\
\end{array}
\right|.
\end{align*}
By multiplying the column $j-1$ by $x_{i_{k}}$ and subtracting the result to the column $j$ in the above matrix, for $j=k-2$ down to $2$, we deduce that 
$$
\tilde{\Delta} =
\left|
\begin{array}{ccccc}
 1 & x_{i_{1}} & \cdots & x_{i_{1}}^{{k-3}} & P^{\{i_{1},i_{k}\}} \\
 1 & x_{i_{2}} & \cdots & x_{i_{2}}^{{k-3}} & P^{\{i_{2},i_{k}\}}\\
 \vdots & \vdots & & \vdots & \vdots \\
  1 & x_{i_{k-1}} & \cdots & x_{i_{k-1}}^{{k-3}} & P^{\{i_{k-1}, i_{k}\}} \\
\end{array}
\right|.
$$
Finally, we obtain
$$
V(x_{i_{1}},\ldots,x_{i_{k}}) P^{\{i_{1},\ldots,i_{k} \}} = (x_{i_{k}}-x_{i_{1}})\cdots(x_{i_{k}}-x_{i_{k-1}})
\left|
\begin{array}{ccccc}
 1 & x_{i_{1}} & \cdots & x_{i_{1}}^{{k-3}} & P^{\{i_{1},i_{k}\}} \\
 1 & x_{i_{2}} & \cdots & x_{i_{2}}^{{k-3}} & P^{\{i_{2},i_{k}\}}\\
 \vdots & \vdots & & \vdots & \vdots \\
  1 & x_{i_{k-1}} & \cdots & x_{i_{k-1}}^{{k-3}} & P^{\{i_{k-1}, i_{k}\}} \\
\end{array}
\right|
$$
and since $\prod_{j=1}^{k-1}(x_{i_{k}}-x_{i_{j}})$ is not a zero divisor in $R[x_{1},\ldots,x_{n}]$ (by Dedekind-Mertens Lemma), it follows that 
$$
V(x_{i_{1}},\ldots,x_{i_{k-1}}) P^{\{i_{1},\ldots,i_{k} \}} = 
\left|
\begin{array}{ccccc}
 1 & x_{i_{1}} & \cdots & x_{i_{1}}^{{k-3}} & P^{\{i_{1},i_{k}\}} \\
 1 & x_{i_{2}} & \cdots & x_{i_{2}}^{{k-3}} & P^{\{i_{2},i_{k}\}}\\
 \vdots & \vdots & & \vdots & \vdots \\
  1 & x_{i_{k-1}} & \cdots & x_{i_{k-1}}^{{k-3}} & P^{\{i_{k-1}, i_{k}\}} \\
\end{array}
\right|.
$$
By repeating this process and using our inductive hypothesis (here on sets of cardinality $3$), we get
$$
V(x_{i_{1}},\ldots,x_{i_{k-2}}) P^{\{i_{1},\ldots,i_{k} \}} = 
\left|
\begin{array}{ccccc}
 1 & x_{i_{1}} & \cdots & x_{i_{1}}^{{k-4}} & P^{\{i_{1},i_{{k-1}},i_{k}\}} \\
 1 & x_{i_{2}} & \cdots & x_{i_{2}}^{{k-4}} & P^{\{i_{2},i_{{k-1}},i_{k}\}}\\
 \vdots & \vdots & & \vdots & \vdots \\
  1 & x_{i_{k-2}} & \cdots & x_{i_{k-1}}^{{k-4}} & P^{\{i_{k-2},i_{{k-1}}, i_{k}\}} \\
\end{array}
\right|.
$$
Continuing this way, we end with the equality
$$ (x_{i_{2}}-x_{i_{1}})  P^{\{i_{1},\ldots,i_{k} \}} = 
V(x_{i_{1}},x_{i_{2}})P^{\{i_{1},\ldots,i_{k} \}} = 
\left|
\begin{array}{cc}
 1 & P^{\{i_{1},i_{3},\ldots,i_{k} \}} \\
 1 & P^{\{i_{2},i_{3},\ldots,i_{k} \}} 
\end{array}
\right|= P^{\{i_{2},i_{3},\ldots,i_{k} \}} -P^{\{i_{1},i_{3},\ldots,i_{k} \}}
$$
which concludes the proof.
\end{proof}

\begin{rem}\label{rem:deg0}
If $n\geq d+1$ then the $d^{\mathrm{th}}$ divided differences are elements in $R$ because there are homogeneous polynomials in $R[x_{1},\ldots,x_{n}]$ of degree $0$. Then, the previous proposition shows that they are all equal : $P^{I}=P^{J}$ for all subsets $I$ and $J$ of $[n]$ such that $|I|=|J|=d+1\leq n$.
\end{rem}
 
\begin{exmp}\label{ex:linform} The more general system of three linear homogeneous polynomials in 3 variables satisfying \eqref{eq:PiPj} is of the form 
$$\left\{
\begin{array}{ccc}
P^{\{1\}} & = & (a+d)x_{1} + bx_{2} + cx_{3} \\
P^{\{2\}} & = & ax_{1} + (b+d) x_{2} + cx_{3}\\
P^{\{3\}} & =  & ax_{1} + bx_{2} + (c+d)x_{3}.
\end{array}
\right.
$$
Some straightforward computations show that $P^{\{1,2\}}=P^{\{1,3\}}=P^{\{2,3\}}=d$ and $P^{\{1,2,3\}}=0$. 
\end{exmp}

The following result is another consequence of Proposition \ref{prop:divdiff} that we record for later use.

\begin{cor}\label{cor:ideal} Let $I$ and $J$ be two subsets of $[n]$ of the same cardinality $r$ with $1\leq r \leq n-1$. Then, the polynomial $P^{I}-P^{J}$ belongs to the ideal of polynomials generated by the $(r+1)^{\mathrm{th}}$ divided differences, that is to say 
$$ P^{I} - P^{J} \in (\ldots, P^{K},\ldots )_{K\subset [n], |K| = r+1}.$$
\end{cor}
\begin{proof} If $|I\cap J|=r-1$ then $P^{I}-P^{J}$ is a multiple of a divided difference $P^{K}$ with $|K| = r+1$ by Proposition \ref{prop:divdiff}. Otherwise, $r\geq 2$, $|I\cap J|<r-1$ and hence there exist $j\in J\setminus I$ and $i\in I\setminus J$ (observe that $i\neq j$ necessarily). Now,
$$ P^{I}-P^{J}=P^{I} - P^{(I\setminus \{i\})\cup \{j\}} + P^{(I\setminus \{i\})\cup \{j\}}- P^{J}$$
where the term $P^{I} - P^{(I\setminus \{i\})\cup \{j\}}$ is a multiple of a divided difference $P^{K}$ with $|K| = r+1$ since 
$|I\cap \left((I\setminus \{i\})\cup \{j\})\right)|=r-1$. So, to prove that $P^{I}-P^{J}$ belongs to the ideal generated by the ${(r+1)}^{\mathrm{th}}$ divided differences amounts to prove that $P^{(I\setminus \{i\})\cup \{j\}}- P^{J}$ belongs to this ideal. But notice that
$|J\cap \left( (I\setminus \{i\})\cup \{j\} \right)|=|I\cap J|+1$. Therefore, one can repeat this operation to reach a cardinality of $r-1$ and from there the conclusion follows. 
\end{proof}

\subsection{Resultant of homogeneous polynomials}\label{subsec:resultant}

Suppose given an integer $n\geq 1$ and a sequence of positive integers $d_1,\ldots,d_{n}$. We consider the \emph{generic}  homogeneous polynomials in the variables $x=(x_1,\ldots,x_n)$ (all assumed to have weight 1) and of degree $d_1,\ldots,d_{n}$ respectively. They are of the form
$$f_i(x_1,\ldots,x_n)=\sum_{|\alpha|=d_i}u_{i,\alpha}x^\alpha, \ \ 
i=1,\ldots,n.$$
The ring ${\UU}:=\ZZ[u_{i,\alpha}: i=1,\ldots,n, |\alpha|=d_i]$ is called the universal ring of coefficients. The polynomials $f_1,\ldots,f_n$ belong to the ring ${C}:={\UU}[x_1,\ldots,x_n]$. Following \cite{Jou91}, the {\it ideal of inertia forms} of these polynomials, i.e.~the ideal $(f_{1},\ldots,f_{n}):(x_{1},\ldots,x_{n})^{\infty}$, is canonically graded and its degree zero part is a
principal ideal of $\UU$. The universal resultant, denoted $\Res$, is then define as the unique generator of this principal ideal such that  
  \begin{equation}\label{eq:normres}
 \Res(x_1^{d_1},\ldots,x_n^{d_n})=1.
 \end{equation}
To define the resultant of any given $n$-uples of homogeneous polynomials in the variables $x_1$,\ldots, $x_n$ (and also to clarify \eqref{eq:normres}) one proceeds as follows. 
Let $S$ be a commutative ring and for all $i=1,\ldots,n$ suppose given a homogeneous polynomial of degree $d_{i}$
    $$g_i=\sum_{|\alpha|=d_i}v_{i,\alpha}x^\alpha \in
    S[x_1,\ldots,x_n]_{d_i}.$$
Then, the resultant of $g_{1},\ldots,g_{n}$ is defined as the image of the universal resultant by the specialization ring morphism  $\theta:{\UU}\rightarrow S:   u_{j,\alpha} \mapsto v_{j,\alpha}$, that is to say 
$$\Res(g_1,\ldots,g_n):=\theta(\Res) \in S.$$
Observe that if $S=\UU$ and $\theta$ is the identity, then the universal resultant $\Res$ is nothing but $\Res(f_{1},\ldots,f_{n})$, which is the notation we will use. If $S$ is a field, then the resultant has the expected geometric interpretation : it vanishes if and only if the polynomials $g_{1},\ldots,g_{n}$ have a common root in the projective space $\PP^{{n-1}}_{\overline{S}}$ (where $\overline{S}$ stands for the algebraic closure of $S$).

\medskip

We now recall briefly some properties of the resultant that we will use in the sequel. For the proofs, we refer the reader to \cite[\S5]{Jou91} (see also  (\cite{Jou97,GKZ94,CLO05}). Let $S$ be any commutative ring and suppose given $g_1,\ldots,g_n$ homogeneous polynomials in the polynomial ring $S[x_1,x_2,\ldots,x_n]$ of positive degree $d_1,\ldots,d_n$ respectively.

\medskip

\paragraph{\emph{Homogeneity}:} for all $i=1,\ldots,n$, $\Res(f_1,\ldots,f_n)$ is homogeneous with respect to the coefficients $(u_{i,\alpha})_{|\alpha|=d_{i}}$ of $f_i$ of degree $d_1\ldots d_n/d_i$.

\medskip

\paragraph{\emph{Permutation of polynomials}:} $\Res(g_{\sigma(1)},\ldots,g_{\sigma(n)})=(\mathcal{E}(\sigma))^{d_1\ldots d_n}\Res(g_1,\ldots,g_n)$ for any permutation $\sigma$ of the set $\{1,\ldots,n\}$ ($\mathcal{E}(\sigma)$ denotes the signature of the permutation $\sigma$).

\medskip

\paragraph{\emph{Elementary transformations}:} $\Res(g_1,\ldots,g_i+\sum_{i\neq j}h_jg_j,\ldots,g_n)=\Res(g_1,\ldots,g_n)$ for any homogeneous polynomials $h_{j}$ of degree $d_{i}-d_{j}$.

\medskip

\paragraph{\emph{Multiplicativity}:} $\Res(g_1'g_1'',g_2,\ldots,g_n)=\Res(g_1',g_2,\ldots,g_n)\Res(g_1'',g_2,\ldots,g_n)$ for any pair of homogeneous polynomials $g_{1}'$ and $g_{1}''$.

\medskip

\paragraph{\emph{Linear change of variables}:} Let $\phi$ be a $n\times n$-matrix with entries in $S$ and denote by $\phi(x)$ the product of the matrix $\phi$ with the column vector $(x_{1}, \ldots,x_{n})^{t}$. Then
$$\Res(g_{1}(\phi(x)),g_{2}(\phi(x)),\ldots,g_{n}(\phi(x)))=\det(\phi)^{d_{1}\cdots d_{n}}\Res(g_{1},\ldots,g_{n}).$$
In particular, the resultant is invariant, up to sign, under permutation of the variables $x_{1},\ldots,x_{n}$.

\medskip

Finally, let us recall quickly the famous \emph{Macaulay formula} that goes back to the work of Macaulay \cite{Mac02} and that is still nowadays a very powerful tool to compute exactly the resultant over a general coefficient ring (all the examples presented in this paper have been computed with this formula). 

Assume we are in the generic setting over the ring $\UU$. Set $\delta:=\sum_{i=1}^{n}(d_{i}-1)$ and denote by $\Mon(n;t)$ the set of all homogeneous monomials of degree $t$ in the $n$ variables $x_{1},\ldots,x_{n}$. if $t\geq \delta+1$ then for any $x^{\alpha} \in \Mon(n;t)$ there exists $i\in \{1,\ldots,n\}$ such that $x_{i}^{d_{i}}$ divides the monomial $x_{\alpha}$. Therefore, in this case we set $i(\alpha):=\min \{ i : x_{i}^{d_{i}} | x^{\alpha} \}$ and we define the square matrix 
$$ \Mt(f_{1},\ldots,f_{n};t)=(m_{\alpha,\beta}) : \Mon(n;t)\times \Mon(n;t) \rightarrow \UU$$
by the formula
$$ \frac{x^{\beta}}{x_{i(\beta)}^{d_{i(\beta)}}}f_{i(\beta)}=\sum_{{|\alpha|=t}}m_{\alpha,\beta}x^{\alpha} \ \textrm{for all } x^{\beta} \in \Mon(n;t).$$
Now, define 
$$\Dod(n;t):=\{ x^{\alpha}\in \Mon(n;t) \textrm{ such that } \exists i\neq j \ : \  x_{i}^{d_{i}}x_{j}^{d_{j}} | x^{\alpha} \}\subset \Mon(n;t)$$ and denote by $\Dt(f_{1},\ldots,f_{n};t)$ the square submatrix of $\Mt(f_{1},\ldots,f_{n};t)$ which is indexed by $\Dod(n;t)$. Now, for any $t\geq \delta +1$ we have the Macaulay formula :
$$ \det(\Mt(f_{1},\ldots,f_{n};t))=\Res(f_{1},\ldots,f_{n})\det(\Dt(f_{1},\ldots,f_{n};t)).$$

\subsection{Discriminant}\label{subsec:disc} Consider the \emph{generic} homogeneous polynomial of degree $d\geq 2$ in $n\geq 2$ variables
$$f(x_1,\ldots,x_n)=\sum_{|\alpha|=d}u_{\alpha}x^\alpha.$$
We denote its universal ring of coefficients ${\UU}:=\ZZ[u_{\alpha}: |\alpha|=d]$, so that $f\in {\UU}[x_1,\ldots,x_n]$. The universal discriminant of $f$, denoted $\Disc(f)$, is defined as the unique element in $\UU$ that satisfies the equality
$$ d^{a(n,d)}\Disc(f)=\Res\left(\frac{\partial f}{\partial x_{1}},\frac{\partial f}{\partial x_{2}},\ldots,\frac{\partial f}{\partial x_{n}} \right)$$
where 
$$a(n,d):=\frac{(d-1)^{n}-(-1)^{n}}{d} \in \ZZ.$$
Similarly to what we have done for the resultant, given a commutative ring $S$ and an homogeneous polynomial of degree $d$
$$g=\sum_{|\alpha|=d} u_{\alpha}x^\alpha \in S[x_{1},\ldots,x_{n}]_{d},$$
its discriminant is denoted by $\Disc(g)$ and is defined as the image of the universal discriminant $\Disc(f)$ by the canonical specialization $\theta:\UU\rightarrow S : u_{\alpha} \mapsto u_{\alpha}$, that is to say
$$\Disc(g)=\theta(\Disc(f)) \in S.$$
With this definition we get a smoothness criterion : If $S$ is an algebraically closed field and $g\neq 0$, then $\Disc(g)=0$ if and only if the hypersurface defined by the polynomial $g$ in 
$\mathrm{Proj}(S[x_{1},\ldots,x_{n}])$ is singular. For a detailed study of the discriminant and its  numerous properties, mostly inherited from the ones of the resultant, we refer the reader to \cite{BuJo12,Dem12,GKZ94} and the references therein. We only point out for future use that the following  property :  the universal discriminant is homogeneous with respect to the coefficient of $f$ of degree $n(d-1)^{n-1}$.
 
\section{Resultant of a $\mathfrak{S}_{n}$-equivariant polynomial system}\label{sec:resultant}

In this section, we consider a polynomial system of $n$ homogeneous equations $F^{\{1\}},\ldots,F^{\{n\}}$ 
in $R[x_{1},\ldots,x_{n}]$, $R$ being an arbitrary commutative ring, of the same degree $d\geq 1$, which is equivariant (see for instance  \cite[\S 4]{Wor94} or \cite[Chapter 1]{DiCa71}) with respect to the canonical actions of the symmetric group $\mathfrak{S}_{n}$ on the variables and polynomials.  
More precisely, we assume that for any integer $i \in \{1,2,\ldots,n\}$ and any permutation $\sigma \in \mathfrak{S}_{n}$
\begin{equation}\label{eq:globsym}
\sigma(F^{\{i\}}):=F^{\{i\}}(x_{\sigma(1)},x_{\sigma(2)},\ldots ,x_{\sigma(n)})=F^{\{\sigma(i)\}}(x_{1},x_{2},\ldots,x_{n}).
\end{equation}
The two following examples suggest that this assumption imposes a decomposition into products of the resultant of  $F^{\{1\}},\ldots,F^{\{n\}}$.

\begin{exmp}\label{eq:jpj} In the case $n=2$ and $d\geq 1$ the polynomial system
$$F^{\{1\}}(x,y):=a_{0}x_{1}^{d}+a_{1}x_{1}^{d-1}x_{2}+\cdots+a_{d}x_{2}^{d}, \ \ F^{\{2\}}(x,y):=F^{\{1\}}(y,x)$$
over the coefficient ring $\ZZ[a_{0},\ldots,a_{d}]$ is the universal  $\mathfrak{S}_{2}$-equivariant polynomial system (any other equivariant system of degree $d$, and with $n=2$, can be obtained as a specialization of this system). One can show (see for instance \cite[Exercice 67]{ApJo}) that there exists an irreducible polynomial $K_{d} \in \ZZ[a_{0},\ldots,a_{d}]$ such that
$$\Res\left( F^{\{1\}},F^{\{2\}} \right)=F^{\{1\}}(1,1)F^{\{1\}}(1,-1)K_{d}^{2}=\left( \sum_{i=0}^{d}a_{i} \right)\left( \sum_{i=0}^{d} (-1)^{i} a_{i} \right)K_{d}^{2}.$$
\end{exmp}

\begin{exmp}\label{ex:d=1} Suppose $n\geq 2$, $d=1$ and $F^{\{i\}}(x_{1},\ldots,x_{n})=ax_{i}+b e_{1}(x_{1},\ldots,x_{n})$, $i=1,\ldots,n$. It is clear that these polynomials satisfy \eqref{eq:globsym}. Moreover, since the resultant of $n$ linear forms in $n$ variables is the determinant of the matrix of their associated linear system, a straightforward computation shows that $$\Res\left(F^{\{1\}},\ldots,F^{\{n\}}\right)=  a^{{n-1}}(a+nb).$$
\end{exmp}

The goal of this section is to prove a general decomposition formula (Theorem \ref{thm:maintheorem}) for the resultant of a  $\mathfrak{S}_{n}$-equivariant homogeneous polynomial system $F^{\{1\}},\ldots,F^{\{n\}}$. We begin this section with some observations on the specialization of divided differences with respect to a given partition of the variables.

\subsection{Divided differences and partitions} \label{subsec:divdiff+part}

A finite sequence $\lambda=(\lambda_{1},\ldots,\lambda_{k})$ of weakly decreasing integers,  i.e.~such that $\lambda_{1} \geq \cdots \geq \lambda_{k}\geq 0$, is called a partition. 
When $\sum_{i=1}^{k}\lambda_{i}=p$ we will say such a $\lambda$ is a partition of $p$, 
and write $\lambda \vdash p$. The number of nonzero $\lambda_{i}$'s is called the length of $\lambda$, 
and will be denoted by $l(\lambda)$.

Given a partition $\lambda \vdash n$, we consider the morphism of polynomial algebras 
\begin{eqnarray}\label{eq:rhol}
 \rho_{\lambda} : R[x_{1},\ldots,x_{n}] & \rightarrow & R[y_{1},\ldots,y_{l(\lambda)}] \\ \nonumber
 F(x_{1},\ldots,x_{n}) & \mapsto & F( \underbrace{y_{1},\ldots,y_{1}}_{\lambda_{1}} , \underbrace{y_{2},\ldots,y_{2}}_{\lambda_{2}}, \ldots,   \underbrace{y_{l(\lambda)},\ldots,y_{l(\lambda)}}_{\lambda_{l(\lambda)}}).
 \end{eqnarray}
where $y_{1},y_{2},\ldots,y_{l(\lambda)}$ are new indeterminates.  Since the polynomials $F^{ \{1\} }, F^{\{2\}}, \ldots, F^{\{n\}}$ satisfy to \eqref{eq:globsym}, they also satisfy to \eqref{eq:PiPj} (observe that Example \ref{ex:linform} shows that systems satisfying \eqref{eq:PiPj} 
are strictly more general than systems satisfying \eqref{eq:globsym}). Indeed, choose a pair of distinct integers $\{i,j\}\in [n]$ and let $\sigma \in \mathfrak{S}_{n}$ be such that $\sigma(k)=k$ if $k\notin \{i,j\}$ and $\sigma(i)=j$, then
\begin{equation}\label{eq:sperho}
 F^{\{i\}}-F^{\{j\}}=F^{\{i\}}-\sigma(F^{\{i\}}) \in (x_{i}-x_{j}).
\end{equation}
Therefore,  the polynomials $F^{ \{1\} }, F^{\{2\}}, \ldots, F^{\{n\}}$ admit divided differences. In addition, 
from their defining equality given in Lemma \ref{lem:divdiff} and from \eqref{eq:globsym}, 
we get that for any subset $\{i_{1},\ldots,i_{k}\}\subset [n]$ and any permutation $\sigma \in \mathfrak{S}_{n}$ 
we have
\begin{equation}\label{eq:perminvariance}
\sigma \left(  F^{\{i_{1},\ldots,i_{k}\}}   \right) = F^{\{  \sigma(i_{1}),\ldots,\sigma(i_{k}) \}}.
\end{equation}

Now, if $\rho_{\lambda}(x_{i})=\rho_{\lambda}(x_{j})$ then \eqref{eq:sperho} implies that
$$\rho_{\lambda}(F^{\{i\}})=\rho_{\lambda}(F^{\{j\}}).$$
So, for any integer $i\in [l(\lambda)]$ we can define without ambiguity the homogeneous polynomial of degree $d$
$$F_{\lambda}^{\{i\}}(y_{1},y_{2},\ldots,y_{l(\lambda)}) := \rho_{\lambda} \left( F^{\{j\}}(x_{1},\ldots,x_{n}) \right) 
\in R[y_{1},\ldots,y_{l(\lambda)}]$$
where $j\in [n]$ is such that $\rho_{\lambda}(x_{j})=y_{i}$. Moreover, these polynomials also satisfy \eqref{eq:PiPj} and hence they also admit divided differences ; we will denote them by $F_{\lambda}^{\{i_{1},\ldots,i_{r}\}}(y_{1},\ldots,y_{l(\lambda)})$ with $\{i_{1},\ldots,i_{r}\}\subset[l(\lambda)]$. From here, a straightforward application of Lemma \ref{lem:divdiff} shows the following property :   Given $I=\{i_{1},\dots,i_{k}\} \subset [n]$, define $J=\{j_{1},\dots,j_{k}\} \subset [l(\lambda)]$ by the equality $\rho_{\lambda}(x_{i_{r}})=y_{j_{r}}$ for all $r \in [k]$. Then, if  $|J|=|I|$ we have
  $$\rho_{\lambda}(F^{I}(x_{1},\dots,x_{n}))=
 F^{J}_{\lambda} (y_{1},\dots,y_{l(\lambda)}).$$

\subsection{The decomposition formula}\label{subsec:mainthm} Before stating the main result of this paper, we need to introduce a last notation. Given a partition $\lambda \vdash n$, its multinomial coefficient is defined as the integer 
\begin{equation}\label{eq:multinomial}
{{n}\choose{\lambda_1, \lambda_2,\ \ldots , \lambda_{l(\lambda)}}} := \frac{n!}{\lambda_1! \lambda_2! \cdots \lambda_{l(\lambda)}!}.
\end{equation}
It counts the number of distributions of $n$ distinct objects to $l(\lambda)$ distinct recipients such that
the recipient $i$ receives exactly $\lambda_i$ objects. In this way of counting, the objects are not ordered inside the boxes, but the boxes are ordered. If we do not want to count the permutations between the boxes having the same number of objects, then we have to divide the above multinomial coefficient by the number of all these permutations. If $s_{j}$ denotes the number of boxes having exactly $j$ objects, $j\in [n]$, then this number of permutations is equal to $\prod_{j=1}^{n} s_{j}!$. Finally, for any partition $\lambda\vdash n$ we define the integer
\begin{equation}\label{eq:mlambda}
m_{\lambda}:=\frac{1}{\prod_{j=1}^{n}s_{j}!} {{n}\choose{\lambda_1, \lambda_2,\ \ldots , \lambda_{l(\lambda)}}}.
\end{equation}

\begin{thm}\label{thm:maintheorem} Assume that $n\geq 2$ and $d\geq 1$. With the above notation, the following equalities hold.  

\noindent $\bullet$ If $d\geq n$ then 
\begin{equation*}
\Res\left(F^{\{1\}},\ldots,F^{\{n\}}\right)=\prod_{\substack{\lambda\vdash n}}
\Res\left( F_{\lambda}^{\{1\}}, F_{\lambda}^{\{1,2\}}, \ldots, F_{\lambda}^{\{1,2,\ldots,l(\lambda)-1\}}, F_{\lambda}^{\{1,2,\ldots,l(\lambda)\}} 
\right)^{m_{\lambda}}.
\end{equation*}

\noindent $\bullet$ If $d<n$ then 
\begin{multline*}
 \Res\left(F^{\{1\}},\ldots,F^{\{n\}}\right)= \\
 \left(F^{\{1,\ldots,d+1\}}\right)^{m_{0}}\times \prod_{ \substack{\lambda\vdash n \\ l(\lambda) \leq d}}
\Res\left( F_{\lambda}^{\{1\}}, F_{\lambda}^{\{1,2\}}, \ldots, F_{\lambda}^{\{1,2,\ldots,l(\lambda)-1\}}, F_{\lambda}^{\{1,2,\ldots,l(\lambda)\}} 
\right)^{m_{\lambda}}
\end{multline*}
where 
\begin{equation*}
 m_{0}:=nd^{n-1}-\sum_{ \substack{\lambda\vdash n \\ l(\lambda) \leq d}} m_{\lambda} 
\left(\sum_{j=1}^{l(\lambda)} \frac{d(d-1)\cdots(d-l(\lambda)+1) }{(d-j+1)} \right).
\end{equation*}
\end{thm}

It is immediate to check that this theorem allows to recover the formulas given in Example \ref{eq:jpj} and Example \ref{ex:d=1}. Before giving its proof, we make some comments on some computational aspects.

 First,  we emphasize that the above formula holds over the universal ring of coefficients of the $\mathfrak{S}_{n}-$equivariant polynomial system $F^{\{1\}},\ldots,F^{\{n\}}$ (over 
$\ZZ$) and it is hence stable under specialization. Our second comment is on the number of terms in these decompositions. It is equal to the cardinality of the set
$$ \{  \ \lambda=(\lambda_{1},\ldots,\lambda_{k}) \vdash d \textrm{ such that }  n \geq \lambda_{1} \geq \lambda_{2} \geq \cdots \geq \lambda_{k} \}$$
that has been extensively studied (we refer the reader to the classical book \cite{MacDonald}). It is important to notice 
that these terms can actually be deduced from a very small number of resultant computations since these resultants are actually also universal with respect to the integers $\lambda_{1}, \lambda_{2},\ldots,\lambda_{l(\lambda)}$ defining a partition, providing $l(\lambda)$ is fixed. Therefore, all the terms in the two decompositions given in Theorem \ref{thm:maintheorem} can be obtained as specializations of only $\min\{n,d\}$ resultant computations. The following example illustrates this property.

\begin{exmp}\label{ex:d=2}  Consider the case $d=2$ and $n\geq 2$ with a polynomial system of the form 
$F^{\{i\}}=\sum_{i=0}^{2}x_{i}^{k}S_{k}$
where $S_{k}$ are symmetric homogeneous polynomials in $x_{1},\ldots,x_{n}$ . More precisely, we consider the polynomials
$$F^{\{i\}}(x_{1},\ldots,x_{n})=ax_{i}^{2}+bx_{i}e_{1}(x_{1},\ldots,x_{n})+ce_{1}(x_{1},\ldots,x_{n})^{2}+de_{2}(x_{1},\ldots,x_{n}), \ \ i=1,\ldots,n.$$
The partition $\lambda=(n)$ yields the factor 
$$\Res\left(F^{\{1\}}_{\lambda}\right)=a+nb+n^{2}c+{{n}\choose{2}}d$$
with multiplicity $m_{\lambda}=1$. 
From Theorem \ref{thm:maintheorem} we know that the other factors come from the partitions of length $2$. They are of the form $\lambda=(m, n-m)$ with $n-1\geq m\geq n-m\geq 1$. The divided difference $F^{\{1,2\}}$ is equal to $a(x_{1}+x_{2})+be_{1}$ and we have
$$\rho_{\lambda}(e_{1})=mx_{1}+(n-m)x_{2}, \ \ \rho_{\lambda}(e_{2})=\binom{m}{2}x_{1}^{2}+m(n-m)x_{1}x_{2}+\binom{n-m}{2}x_{2},$$
$$F_{\lambda}^{\{1,2\}}=\rho_{\lambda}\left(F^{\{1,2\}}\right)=a(x_{1}+x_{2})+b\rho_{\lambda}(e_{1})=a(x_{1}+x_{2})+b(mx_{1}+(n-m)x_{2}).$$
Therefore, such a partition $\lambda=(m,n-m)$ yields the factor
\begin{multline}\label{eq:d=2}
\Res\left( F^{\{1\}}_{(m,n-m)},F^{\{1,2\}}_{(m,n-m)} \right) = 
a{b}^{2}nm+2\,d{m}^{2}ab-1/2\,dm{b}^{2}{n}^{2}+1/2\,d{m}^{2}{b}^{2}n-2
\,dmn{a}^{2}-4\,cmn{a}^{2} \\ 
-2\,dmabn+1/2\,d{n}^{2}{a}^{2}+2\,d{m}^{2}{a
}^{2}+{a}^{2}bn-1/2\,dn{a}^{2}+c{n}^{2}{a}^{2}+4\,c{m}^{2}{a}^{2}-a{b}
^{2}{m}^{2}+{a}^{3}
\end{multline}
which is computed as the determinant of a $3\times 3$ Sylvester matrix. 
To summarize, if $n=2$ (and $d=2$) we get 
\begin{multline*}
 \Res(F^{\{1\}},F^{\{2\}})=\Res\left(F^{\{1\}}_{(2)}\right)\Res\left(F^{\{1\}}_{(1,1)},F^{\{1,2\}}_{(1,1)}\right)= 
(a+2b+4c+d)  \left( a+b \right) ^{2}  \left( a-d \right)
\end{multline*}
where $\Res\left(F^{\{1\}}_{(1,1)},F^{\{1,2\}}_{(1,1)}\right)$ is obtained by specialization of \eqref{eq:d=2}. 
If $n>2$ (and $d=2$) then it is easy to check that $F^{\{1,2,3\}}=a$. Therefore, if $n=2k+1$, $k$ being a positive integer, then
$$ \Res(F^{1},F^{2})=(a)^{m_{0}}\left( a+nb+n^{2}c+{{n}\choose{2}}d \right) 
 \prod_{m=k+1}^{n-1} \Res\left( F^{\{1\}}_{(m,n-m)},F^{\{1,2\}}_{(m,n-m)} \right)^{\frac{n!}{m!(n-m)!}}
$$
where the resultants in this formula are again given by \eqref{eq:d=2} and 
$$m_{0}=n2^{{n-1}}-1-3\sum_{m=k+1}^{{n-1}}\frac{n!}{m!(n-m)!}.$$
If $n=2k$ with $k>1$ then
\begin{multline*}
 \Res(F^{1},F^{2})=(a)^{m_{0}}\left( a+nb+n^{2}c+{{n}\choose{2}}d \right) \Res\left( F^{\{1\}}_{(k,k)},F^{\{1,2\}}_{(k,k)} \right)^{\frac{1}{2}\frac{n!}{(k!)^{2}}}\times \\
 \prod_{m=k+1}^{n-1} \Res\left( F^{\{1\}}_{(m,n-m)},F^{\{1,2\}}_{(m,n-m)} \right)^{\frac{n!}{m!(n-m)!}}
\end{multline*}
where the resultants in this formula are always given by \eqref{eq:d=2} and
$$m_{0}=n2^{{n-1}}-1-\frac{3}{2}\frac{n!}{(k!)^{2}} -3\sum_{m=k+1}^{{n-1}}\frac{n!}{m!(n-m)!}.$$
Before closing this example, we emphasize that the resultants appearing in Theorem \ref{thm:maintheorem} are not always (geometrically) irreducible polynomials. For instance, in the case where $n=2k$ is an even integer, we have
\begin{equation}\label{eq:m=n/2}
\Res\left( F^{\{1\}}_{(k,k)},F^{\{1,2\}}_{(k,k)} \right) = (a+bk)^{2}(a-dk)^{2}.
\end{equation}
However, we notice that $\Res(F^{\{1\}}_{\lambda})$ is obviously always irreducible (in the universal setting).
\end{exmp}

From a geometric point of view, Theorem \ref{thm:maintheorem} shows that the algebraic polynomial system 
$$\{F^{\{1\}}=0,\ldots,F^{\{n\}}=0\}$$ 
can be split into the smaller algebraic systems 
\begin{equation}\label{eq:systnoninvariant}
 \{F^{\{1\}}_{\lambda}=0,\ldots,F^{\{1,\ldots,l(\lambda)\}}_{\lambda}=0\}, \  \lambda \vdash n,  \ l(\lambda )\leq d
\end{equation}
with multiplicity $m_{\lambda}$, respectively.  For each given partition $\lambda=(\lambda_{1},\lambda_{2},\ldots,\lambda_{l(\lambda)})$, the algebraic systems \eqref{eq:systnoninvariant} correspond to particular configurations of the roots of the initial system, namely the roots whose coordinates can be grouped into $l(\lambda)$ blocks of size $\lambda_{1}, \ldots, \lambda_{l(\lambda)}$ respectively, up to permutations.

\subsection{Proof of Theorem \ref{thm:maintheorem}}\label{subsec:proof} 

We begin by splitting the resultant of the  $F^{{\{i\}}}$'s into several factors by means of their divided differences. This process can be divided into steps where we increase iteratively the order of the divided differences. Thus, in the first step we make use of the first order divided differences and write
\begin{multline}\label{eq:split1}
\Res\left(F^{\{1\}},F^{\{2\}},\ldots,F^{\{n\}}\right)=\\ 
\pm \Res\left(F^{\{1\}}, (x_{1}-x_{2})F^{\{1,2\}},(x_{1}-x_{3})F^{\{1,3\}},\ldots,(x_{1}-x_{n})F^{\{1,n\}}\right). 
\end{multline}
The divided differences $F^{\{1,j\}}$ are of degree $d-1$. If $d-1=0$ then they are all equal to the same constant by Remark \ref{rem:deg0} and it is straightforward to check that we get the claimed formula in this case, that is to say 
$$ \Res\left(F^{\{1\}},F^{\{2\}},\ldots,F^{\{n\}}\right)=\left(F^{{\{1,2\}}}\right)^{{n-1}}\Res \left(F^{\{1\}}_{(n)}\right)=\left(F^{{\{1,2\}}}\right)^{{n-1}} F^{\{1\}}(1,1,\ldots,1).$$
 If $d-1>0$, then \eqref{eq:split1} shows that the resultant of the  $F^{{\{i\}}}$'s splits into $2^{n-1}$ factors by using the multiplicativity property of the resultant : for each polynomial $(x_{1}-x_{j})F^{\{1,j\}}$, $j=2,\ldots,n$, there is a choice between  $(x_{1}-x_{j})$ and  the divided difference $F^{\{1,j\}}$. Thus, these factors are in bijection with the subsets of $[n]$ that contain 1. If $I_{1}=\{1,i_{2},i_{3},\ldots,i_{n-k+1}\}\subset [n]$ is such a subset, then the corresponding factor is simply
$$\pm \Res\left(  
F^{{\{1\}}},F^{\{1,j_{1}\}},F^{\{1,j_{2}\}},\ldots,F^{\{1,j_{k}\}},x_{1}-x_{i_{2}},x_{1}-x_{i_{3}},\ldots,x_{1}-x_{i_{n-k+1}}
\right)$$
where $\{j_{1},\ldots,j_{k-1}\}=[n]\setminus I_{1}$. Moreover, by the specialization property of the resultant this factor is equal to 
\begin{equation}\label{eq:step1}
 \pm \Res\left( 
F_{1}^{{\{1\}}},F_{1}^{\{1,2\}},F_{1}^{\{1,3\}},\ldots,F_{1}^{\{1,k\}}
\right)
\end{equation}
where we set $F_{1}^{\{1,r\}}:=\rho_{1}(F^{{\{1,j_{r}\}}})$, $\rho_{1}$ being a specialization map defined by 
\begin{eqnarray*}
 \rho_{1}: k[x_{1},\ldots,x_{n}] & \rightarrow & k[x_{1},\ldots,x_{k}]\\
 x_{j}, \, j\in I_{1}  & \mapsto & x_{1}\\
 x_{j_{r}}, \, r=1,\ldots,k-1 & \mapsto & x_{r+1}. 
\end{eqnarray*}
Roughly speaking, this amounts to put all the variables $x_{j}$, $j\in I_{1}$, in the ``same box'' and to renumber the other variables from 2 to $k$. 

Now, one can proceed to the second step by introducing the second order divided differences. For that purpose, we start from the factor \eqref{eq:step1} obtained at the end of the previous step. If $k\leq 2$ then we actually do nothing and the splitting of this factor stops here. Otherwise, If $k>2$ then we can proceed exactly as in the first step : Since 
$$(x_{2}-x_{j})F_{1}^{\{1,2,j\}}=F_{1}^{\{1,2\}}-F_{1}^{\{1,j\}}, \ j=3,\ldots,k,$$
we get 
\begin{multline*}
 \Res\left( 
F_{1}^{{\{1\}}},F_{1}^{\{1,2\}},F_{1}^{\{1,3\}},\ldots,F_{1}^{\{1,k\}}\right)=\\
\pm  \Res\left(F^{\{1\}}, F^{\{1,2\}},(x_{2}-x_{3})F^{\{1,2,3\}}, (x_{2}-x_{4})F^{\{1,2,4\}}, \ldots,(x_{2}-x_{k})F^{\{1,2,k\}}\right). 
\end{multline*}
So, we are exactly in the same setting as in the previous step and hence we split this factor similarly. As a result, the factors we obtain are in bijection with  subsets $I_{2}$ of $[n]$ that contain 2 but not 1. After this second step is completed, then one can continue to the third step, and so on. This splitting process stops for a given factor if either it involves divided differences of distinct orders or either the order of some divided differences is higher than the degree $d$. 

\medskip

In summary, the above process shows that the resultant $\Res\left(F^{\{1\}},F^{\{2\}},\ldots,F^{\{n\}}\right)$ splits into factors that are in bijection with ordered collections of subsets $(I_{1},\ldots,I_{k})$ that satisfy the following three conditions :
\begin{itemize}
 \item $1\leq k \leq \min\{d,n\} \textrm{ and }  \emptyset \neq I_{j} \subset [n] \textrm{ for all } j\in [k],$ 
 \vspace{.2em}
 \item $I_{1} \coprod I_{2} \coprod \ldots \coprod I_{k}=[n] \textrm{ (disjoint union, so this is a partition of } [n]),$
  \vspace{.2em}
 \item $1=\min(I_{1})< \min(I_{2})<\cdots < \min(I_{k}).$
\end{itemize}
\begin{defn}
A collection of subsets $(I_{1},\ldots,I_{k})$ satisfying to the three above conditions will be called an \emph{admissible partition} (of $[n]$).
\end{defn}
Given an admissible partition  $(I_{1},\ldots,I_{k})$, we define the specialization map 
\begin{eqnarray*}
\rho_{(I_{1},\ldots,I_{k})}:k[x_{1},\ldots,x_{n}] & \rightarrow & k[x_{1},\ldots,x_{k}] \\
 x_{r}, r \in I_{s} & \mapsto & x_{s} 
\end{eqnarray*}
and the polynomials $F_{(I_{1},\ldots,I_{k})}^{\{1,2,\ldots,r\}}:=\rho_{(I_{1},\ldots,I_{k})}(F^{\{1,i_{2},\ldots,i_{r}\}})$, $r=1,\ldots,k$, where we set 
$$i_{1}:=1=\min(I_{1})< i_{2}:=\min(I_{2})<\cdots < i_{k}:=\min(I_{k}).$$
Then, the  factor of the resultant of the $F^{\{i\}}$'s corresponding to the admissible partition $(I_{1},\ldots,I_{k})$  is given by
$$R_{(I_{1},\ldots,I_{k})}:= \Res\left( 
F_{(I_{1},\ldots,I_{k})}^{\{1\}}, F_{(I_{1},\ldots,I_{k})}^{\{1,2\}}, \ldots, F_{(I_{1},\ldots,I_{k})}^{\{1,2,\ldots,k\}}
\right).
$$
Therefore, we proved that
\begin{equation}\label{eq:intermproof}
 \Res\left(F^{\{1\}},F^{\{2\}},\ldots,F^{\{n\}}\right)=\pm \left(F^{\{1,\ldots,d+1\}} \right)^{\mu} \times \prod_{{(I_{1},\ldots,I_{k})}} R_{(I_{1},\ldots,I_{k})}
\end{equation}
where the product runs over all admissible partitions of $[n]$ and $\mu$ is an integer. Moreover, $\mu > 0$ if and only only if $n>d$. 

\medskip

Now, we define an equivalence relation $\sim$ on the set of admissible partitions of $[n]$. Given two admissible partitions $(I_{1},\ldots,I_{k})$ and $(J_{1},\ldots,J_{k'})$, we set 
$$ (I_{1},\ldots,I_{k}) \sim (J_{1},\ldots,J_{k'}) \Leftrightarrow 
\begin{cases}
 k=k' \textrm{ and }\\
 \exists \, \sigma \in \mathfrak{S}_{k} \textrm{ such that } |I_{l}|=|J_{\sigma(l)}| \textrm{ for all } l\in[k].
\end{cases}$$
It is straightforward to check that this binary relation is reflexive, symmetric and transitive so that it defines an equivalence relation. We denote by $[(I_{1},\ldots,I_{k})]$ its equivalence classes. Consider the admissible partitions $(L_{1},\ldots,L_{k})$ such that 
\begin{equation}\label{eq:defL}
 l_{1}:=|L_{1}|\geq l_{2}:=|L_{2}|\geq \ldots l_{k}:=|L_{k}| \ \textrm{ and }
\end{equation}
$$L_{j}:=\left\{1+\sum_{i=1}^{j-1} l_{i},2+\sum_{i=1}^{j-1} l_{i},\ldots,\sum_{i=1}^{j}l_{i}\right\} \textrm{ for all } j\in [k].$$
Obviously, there is exactly one such admissible partition in each equivalent class of $\sim$.  Moreover, these admissible partitions are in bijection with 
the partitions $\lambda \vdash n$ of length $k$ by setting $\lambda:=(l_{1},l_{2},\ldots,l_{k})\vdash n$.
As a consequence, we deduce that there is a bijection between  the equivalence classes of $\sim$ and the partitions $\lambda \vdash n$ of length $k$ and we write
$$ [\lambda]:=[(I_{1},\ldots,I_{k})] = [(L_{1},\ldots,L_{k})].$$

\begin{lem}\label{lem:proof} Let $\lambda$ be a partition of $n$, then the cardinality of the equivalence class $[\lambda]$ is $m_{\lambda}$.  
\end{lem}
\begin{proof}
Let $\lambda$ be a partition of $n$ and consider the equivalent class  $[\lambda]$. The multinomial coefficient \eqref{eq:multinomial} counts the different ways of filling $k=\l(\lambda)$ boxes $J_{1},\ldots,J_{k}$ with $\lambda_{j}$ elements in the box $J_{j}$. These choices take into account the order between the boxes, but not inside the boxes. These boxes $J_{j}$ can obviously be identified with subsets of $[n]$. Moreover, there exists a unique permutation $\sigma \in \mathfrak{S}_{k}$ such that 
$$1=\min(J_{\sigma(1)}) < \min(J_{\sigma(2)}) < \cdots \min(J_{\sigma(k)})$$
and hence such that the collection of subsets $(J_{\sigma(1)},J_{\sigma(2)},\ldots,J_{\sigma(k)})$ is an admissible partition. Therefore, any choice for filling the boxes $J_{1},\ldots,J_{k}$ can be associated  to a factor in the decomposition. Conversely, such a factor is associated to an admissible partition $(I_{1},\ldots,I_{k})$, but there are possibly several choices, i.e.~permutations in $\mathfrak{S}_{k}$, that give a way of filling the  boxes $J_{1},\ldots,J_{k}$: it is possible to permute boxes that have the same cardinality. Therefore, we conclude that the cardinality of the equivalent class represented by a partition $\lambda \vdash n$ is exactly $m_{\lambda}$. 
\end{proof}

The following result shows that admissible partitions that are equivalents give the same factor, up to sign, in the splitting process. 

\begin{prop}\label{prop:proof}  Let $\lambda$ be a partition of $n$. Then, for any admissible partition $(I_{1},\ldots,I_{k})$ such that  $[\lambda]=[(I_{1},\ldots,I_{k})]$,  
$$R_{(I_{1},\ldots,I_{k})}=\pm 
\Res\left( F_{\lambda}^{\{1\}}, F_{\lambda}^{\{1,2\}}, \ldots, F_{\lambda}^{\{1,2,\ldots,l(\lambda)-1\}}, F_{\lambda}^{\{1,2,\ldots,l(\lambda)\}} 
\right).$$
\end{prop}
\begin{proof} Let $(I_{1},\ldots,I_{k})$ be an admissible partition and set 
$$i_{1}:=1=\min(I_{1})< i_{2}:=\min(I_{2})<\cdots < i_{k}:=\min(I_{k}).$$
Its corresponding factor in the splitting process is nothing but the resultant, up to sign, of the following list of $n$ polynomials in the $n$ variables 
$x_{1},\ldots,x_{n}$:
\begin{equation}\label{eq:listI}
 F^{\{1\}}, F^{\{1,i_{2}\}}, \ldots, F^{\{1,i_{2},\ldots,i_{k}\}},
\left\{x_{i_{1}}-x_{r}\right\}_{r\in I_{1}\setminus \{1\}}, 
\ldots,
\left\{x_{i_{k}}-x_{r}\right\}_{r\in I_{k}\setminus \{i_{k}\}}.
\end{equation}
Now, let $(J_{1},J_{2},\ldots,J_{k})$ be another admissible partition such that $[(I_{1},\ldots,I_{k})]=[(J_{1},J_{2},\ldots,J_{k})]$ and set 
$$j_{1}:=1=\min(J_{1})< j_{2}:=\min(J_{2})<\cdots < j_{k}:=\min(J_{k}).$$ 
The corresponding factor of $(J_{1},J_{2},\ldots,J_{k})$ can be described similarly as the resultant, up to sign, of the polynomials
\begin{equation}\label{eq:listJ}
 F^{\{1\}}, F^{\{1,j_{2}\}}, \ldots, F^{\{1,j_{2},\ldots,j_{k}\}},
\left\{x_{j_{1}}-x_{r}\right\}_{r\in J_{1}\setminus \{1\}}, 
\ldots,
\left\{x_{j_{k}}-x_{r}\right\}_{r\in J_{k}\setminus \{j_{k}\}}.
\end{equation}

First, observe that it is sufficient to prove that $R_{(I_{1},\ldots,I_{k})}=\pm R_{(J_{1},\ldots,J_{k})}$ by assuming that $|I_{\sigma(l)}|=|J_{l}|$ for all $l\in [k]$ where $\sigma$ is an elementary transposition (a permutation which exchanges two succesive elements and keeps all the others fixed) in $\mathfrak{S}_{k}$. This is because $\mathfrak{S}_{k}$ is generated by the elementary transpositions and because of the transitivity of $\sim$. So, let $s\in [k-1]$ and assume that 
$$  |I_{s}|=|J_{{s+1}}|, \ |I_{s+1}|=|J_{s}|    \textrm{ and }   |I_{l}|=|J_{l}| \textrm{ for all } l\in [k]\setminus\{s,s+1\}.$$
Let us choose a permutation $\tau \in \mathfrak{S}_{n}$  such that 
$$\begin{cases}
 \tau(I_{l})=J_{l} \textrm{ and } \tau(i_{l})=j_{l} \textrm{ for all } l\in [k], \\
 \tau(I_{s})=J_{s+1} \textrm{ and } \tau(i_{s})=j_{s+1}, \\
 \tau(I_{s+1})=J_{s} \textrm{ and } \tau(i_{s+1})=j_{s}.
\end{cases}$$
By the property \eqref{eq:perminvariance},  the application of $\tau$ on the list of polynomials \eqref{eq:listI} returns the following list of polynomials
\begin{multline}\label{eq:list3}
 F^{\{1\}}, F^{\{1,j_{2}\}}, \ldots,  F^{\{1,j_{2},\ldots,j_{s-1},j_{s+1}\}},F^{\{1,j_{2},\ldots,j_{s-1},j_{s},j_{s+1}\}},\ldots, F^{\{1,j_{2},\ldots,j_{k}\}}, \\
\left\{x_{j_{1}}-x_{r}\right\}_{r\in J_{1}\setminus \{1\}}, \ldots, 
\left\{x_{j_{s-1}}-x_{r}\right\}_{r\in J_{s-1}\setminus \{j_{s-1}\}},
\left\{x_{j_{s+1}}-x_{r}\right\}_{r\in J_{s+1}\setminus \{j_{s+1}\}}, \\
\left\{x_{j_{s}}-x_{r}\right\}_{r\in J_{s}\setminus \{j_{s}\}},
\ldots,
\left\{x_{j_{k}}-x_{r}\right\}_{r\in J_{k}\setminus \{j_{k}\}}.
\end{multline}
By the invariance, up to sign, of the resultant under permutations of polynomials and variables (see \S \ref{subsec:resultant}), we get that the resultant of the list of polynomials \eqref{eq:listI}, i.e.~$R_{(I_{1},\ldots,I_{k})}$, is equal to the resultant of the list of polynomials \eqref{eq:list3} up to sign. Now, by Proposition \ref{prop:divdiff}, we have
$$F^{\{1,j_{2},\ldots,j_{s-1},j_{s}\}}=F^{\{1,j_{2},\ldots,j_{s-1},j_{s+1}\}} +(x_{j_{s}}-x_{j_{s+1}})F^{\{1,j_{2},\ldots,j_{s-1},j_{s},j_{s+1}\}}$$ 
so that the resultant of the polynomials \eqref{eq:list3} is equal, up to sign, to the resultant of the polynomials \eqref{eq:listJ}, i.e.~$R_{(J_{1},\ldots,J_{k})}$, by invariance of the resultant under the above elementary transformation and permutations of polynomials. Therefore, we have proved that $R_{(I_{1},\ldots,I_{k})}=\pm R_{(J_{1},\ldots,J_{k})}$.

Finally, to conclude the proof, let $(L_{1},\ldots,L_{k})$ be the particular representative of the class $[\lambda]=[(I_{1},\ldots,I_{k})]$ as defined in \eqref{eq:defL}. Then, it is clear by the definitions that $\rho_{(L_{1},\ldots,L_{k})}=\rho_{\lambda}$ and that
$$R_{(L_{1},\ldots,L_{k})}= 
\Res\left( F_{\lambda}^{\{1\}}, F_{\lambda}^{\{1,2\}}, \ldots, F_{\lambda}^{\{1,2,\ldots,l(\lambda)-1\}}, F_{\lambda}^{\{1,2,\ldots,l(\lambda)\}} 
\right).$$
\end{proof}

The comparison of \eqref{eq:intermproof}, Lemma \ref{lem:proof} and Proposition \ref{prop:proof} shows that  if $d\geq n$ then
\begin{equation*}
\Res\left(F^{\{1\}},\ldots,F^{\{n\}}\right)= \pm \prod_{\substack{\lambda\vdash n}}
\Res\left( F_{\lambda}^{\{1\}}, F_{\lambda}^{\{1,2\}}, \ldots, F_{\lambda}^{\{1,2,\ldots,l(\lambda)-1\}}, F_{\lambda}^{\{1,2,\ldots,l(\lambda)\}} 
\right)^{m_{\lambda}}
\end{equation*}
 and if  $n>d$ then
\begin{multline*}
 \Res\left(F^{\{1\}},\ldots,F^{\{n\}}\right)= \\
\pm  \left(F^{\{1,\ldots,d+1\}}\right)^{\mu}\prod_{ \substack{\lambda\vdash n \\ l(\lambda) \leq d}}
\Res\left( F_{\lambda}^{\{1\}}, F_{\lambda}^{\{1,2\}}, \ldots, F_{\lambda}^{\{1,2,\ldots,l(\lambda)-1\}}, F_{\lambda}^{\{1,2,\ldots,l(\lambda)\}} 
\right)^{m_{\lambda}}.
\end{multline*}

To determine the integer $\mu$, we compare the degrees with respect to the coefficients of the $F^{\{i\}}$'s. The resultant on the left side is homogeneous of degree $d^{n-1}$ with respect to the coefficients of each polynomial $F^{\{i\}}$, so it is homogeneous of degree $nd^{n-1}$ with respect to  the coefficients of all the polynomials $F^{\{i\}}$, $i=1,\ldots,n$. Given a partition $\lambda \vdash n$, $l(\lambda)\leq d$, the polynomial $F_{\lambda}^{\{1,2,\ldots,j\}}$, $1\leq j \leq l(\lambda)$ is of degree $d-j+1$ by Lemma \ref{lem:divdiff}. Therefore, the resultant associated to this partition $\lambda$ is homogeneous with respect to the coefficients of the $F^{\{i\}}$'s of degree
$$ \sum_{j=1}^{l(\lambda)} \frac{d(d-1)\cdots (d-l(\lambda)+1)}{d-j+1}.$$
Finally, since $F^{\{1,2,\ldots,d+1\}}$ is homogeneous of degree one in the coefficient of the $F^{\{i\}}$'s (see the defining equality in Lemma \ref{lem:divdiff}), we deduce that $\mu$ is equal to the integer $m_{0}$ defined in the statement of Theorem \ref{thm:maintheorem}. 

\begin{rem}
If we apply the above degree counting in the case $d\geq n$, we get the following combinatorial formula for which we do not know if it is known: if $d\geq n$ then 
$$nd^{n-1}=\sum_{\lambda\vdash n} m_{\lambda}
\left(
\sum_{j=1}^{l(\lambda)}
\frac{d(d-1)\cdots (d-l(\lambda)+1)}{d-j+1}
\right).$$
\end{rem}

\medskip

To conclude the proof of Theorem \ref{thm:maintheorem}, it remains to determine the sign $\pm$ that occurs in the two formulas. For that purpose, we examine the specialization of these formulas to the case where $F^{\{i\}}=x_{i}^{d}$, $i=1,\ldots,n$. First, it follows from \eqref{eq:normres} that the resultant of the $F^{\{i\}}$'s is equal to 1. Now, given any partition $\lambda\vdash n$, it is straightforward to check that $F_{\lambda}^{\{1\}}=x_{1}^{d}$. Then applying iteratively Proposition \ref{prop:divdiff} from $j=1$ to $j=l(\lambda)$, it follows that 
$$F_{\lambda}^{\{1,2,\ldots,j\}}=x_{j}^{d} \mod (x_{1},\ldots,x_{j-1}), \ \ j=1,\ldots,l(\lambda).$$
From here, using the multiplicativity property of the resultant and its invariance under elementary transformations, we deduce that all the resultants associated to a partition $\lambda$ specialize to 1. Finally, by Lemma \ref{lem:divdiff} it appears that $F^{\{1,\ldots,d+1\}}$ also specializes to 1 in the case $n>d$ and this concludes the proof of Theorem \ref{thm:maintheorem}.

\subsection{Averaging over the divided differences of the same order} \label{subset:averaging}
Since the polynomials $F^{ \{1\} }, F^{\{2\}}, \ldots, F^{\{n\}}$ satisfy to the property \eqref{eq:perminvariance}, it follows that for any integer $k\in [n]$ 
$$\sum_{\substack{I\subset [n], \, |I|=k}} F^{I}=\sum_{\substack{I\subset [n], \, |I|=k}} F^{\sigma(I)}=\sum_{\substack{I\subset [n], \, |I|=k}} \sigma(F^{I})=\sigma\left(\sum_{\substack{I\subset [n], \, |I|=k}} F^{I}\right).$$
Therefore, for all $k\in [n]$, the polynomial  $\sum_{\substack{I\subset [n], \, |I|=k}} F^{I}$ is symmetric. Such a property is useful for applying various polynomial system solving methods (see e.g.~\cite{FS12}). In general, this property is no longer true if we consider $F^{I}_{\lambda}$ instead of $F^{I}$ (except for the case $\lambda=(1,1,\ldots,1)$ which is the case investigated in \cite[\S 3.3]{FS12}). Nevertheless, it is possible to reformulate Theorem \ref{thm:maintheorem} by means of these sums of divided differences of the same order.

\begin{prop}\label{prop:sym} Taking again the notation of Theorem \ref{thm:maintheorem}, then for any partition $\lambda \vdash n$ such that 
$l(\lambda)\leq \min \{d,n\}$ we have
\begin{multline*}
 \Res\left( \sum_{I\subset [l(\lambda)], \, |I|=1} F_\lambda^{I}, \sum_{I\subset [l(\lambda)], \, |I|=2} F_\lambda^{I}, \ldots, \sum_{I\subset [l(\lambda)], \, |I|=l(\lambda)-1} F_\lambda^{I},
 F_\lambda^{\{1,2,\ldots,l(\lambda\})}\right)= \\
 \left(\prod_{k=1}^{l(\lambda)-1} \binom{l(\lambda)}{k}^{\frac{d(d-1)(d-2)\cdots(d-l(\lambda)+1)}{d-k+1}}
 \right)
 \Res\left( F_{\lambda}^{\{1\}}, F_{\lambda}^{\{1,2\}}, \ldots, F_{\lambda}^{\{1,2,\ldots,l(\lambda)-1\}}, F_{\lambda}^{\{1,2,\ldots,l(\lambda)\}} 
\right).
 \end{multline*}
\end{prop}
\begin{proof} For any subset $I\subset [l(\lambda)]$ such that $|I|=l(\lambda)-1$, Corollary \ref{cor:ideal} shows that
\begin{equation}\label{eq:redconstant}
 F_{\lambda}^{I}=F_{\lambda}^{\{1,2,\ldots,l(\lambda)-1\}} \mod \left( F_{\lambda}^{\{1,2,\ldots,l(\lambda\})}  \right)
\end{equation}
from we deduce that
$$\sum_{I\subset [l(\lambda)], \, |I|=l(\lambda)-1} F_\lambda^{I} = l(\lambda)\, F_{\lambda}^{\{1,2,\ldots,l(\lambda)-1\}} \mod \left( F_{\lambda}^{\{1,2,\ldots,l(\lambda\})}  \right).$$
In the same way, for any subset $I\subset [l(\lambda)]$ such that $|I|=l(\lambda)-2$, 
Corollary \ref{cor:ideal}  shows that
$$F_{\lambda}^{I}=F_{\lambda}^{\{1,2,\ldots,l(\lambda)-2\}} \mod \left(  \left\{ F_{\lambda}^{I}  \right\}_{|I|=l(\lambda)-1},  F_{\lambda}^{\{1,2,\ldots,l(\lambda\})}  \right).$$
Using \eqref{eq:redconstant}, this equality can be simplified to give 
$$F_{\lambda}^{I}=F_{\lambda}^{\{1,2,\ldots,l(\lambda)-2\}} \mod \left(  F_{\lambda}^{\{1,2,\ldots,l(\lambda)-1 \}}  ,  F_{\lambda}^{\{1,2,\ldots,l(\lambda\})}  
\right).$$
We deduce that 
$$\sum_{I\subset [l(\lambda)], \, |I|=l(\lambda)-2} F_\lambda^{I} = \binom{l(\lambda)}{2}\, F_{\lambda}^{\{1,2,\ldots,l(\lambda)-2\}} \mod \left(  F_{\lambda}^{\{1,2,\ldots,l(\lambda)-1 \}}, F_{\lambda}^{\{1,2,\ldots,l(\lambda\})}  \right).$$
By applying iteratively this method, we obtain for all $k=1,\ldots,l(\lambda)-1$ the equality
\begin{equation*}\label{eq:reducsym}
 \sum_{I\subset [l(\lambda)], \, |I|=l(\lambda)-k} F_\lambda^{I} = \binom{l(\lambda)}{k}\, F_{\lambda}^{\{1,2,\ldots,l(\lambda)-k\}} \mod \left(  F_{\lambda}^{\{1,2,\ldots,l(\lambda)-k+1 \}}, \ldots, F_{\lambda}^{\{1,2,\ldots,l(\lambda\})}  \right).
\end{equation*}
From these equalities, the invariance of the resultant under elementary transformations yields the equality (proceed from the right to the left)
\begin{multline*}
 \Res\left( \sum_{I\subset [l(\lambda)], \, |I|=1} F_\lambda^{I}, \sum_{I\subset [l(\lambda)], \, |I|=2} F_\lambda^{I}, \ldots, \sum_{I\subset [l(\lambda)], \, |I|=l(\lambda)-1} F_\lambda^{I},
 F_\lambda^{\{1,2,\ldots,l(\lambda\})}\right)= \\
 \Res\left( \binom{l(\lambda)}{1} F_{\lambda}^{\{1\}}, \binom{l(\lambda)}{2}F_{\lambda}^{\{1,2\}}, \ldots, \binom{l(\lambda)}{1} F_{\lambda}^{\{1,2,\ldots,l(\lambda)-1\}}, F_{\lambda}^{\{1,2,\ldots,l(\lambda)\}} 
\right).
\end{multline*}
 Now, the claimed result follows from the multi-homogeneity of the resultant since the polynomials $F_{\lambda}^{I}$ are homogeneous of degree $d-|I|+1$.
\end{proof}

As a consequence of the proof of this proposition, we see that the big constant factor can be removed by taking averages in the sums of divided differences of the same order. More precisely, assume that the coefficient ring contains the rational numbers and set
$$\mathcal{F}_{\lambda}^{(k)} :=\frac{1}{\binom{l(\lambda)}{k}}\sum_{\substack{I\subset [l(\lambda)], \, |I|=k}} F^{I}_{\lambda}.$$
Then, we obtain the equality
  \begin{equation*}
 \Res\left( \mathcal{F}_{\lambda}^{(1)} , \mathcal{F}_{\lambda}^{(2)} , \ldots, \mathcal{F}_{\lambda}^{(l(\lambda))} \right)=
   \Res\left(  F_{\lambda}^{\{1\}}, F_{\lambda}^{\{1,2\}}, \ldots,  F_{\lambda}^{\{1,2,\ldots,l(\lambda)-1\}}, F_{\lambda}^{\{1,2,\ldots,l(\lambda)\}} 
\right).
\end{equation*}

\begin{exmp} Taking again the notation of Example \ref{ex:d=2}, a direct computation shows that
$$\Res\left( F^{\{1\}}_{(m,n-m)}+F^{\{2\}}_{(m,n-m)},F^{\{1,2\}}_{(m,n-m)} \right) = 2 \, \Res\left( F^{\{1\}}_{(m,n-m)},F^{\{1,2\}}_{(m,n-m)} \right).$$
\end{exmp}

\section{Discriminant of a homogeneous symmetric polynomial}\label{sec:discriminant}

The discriminant of a homogeneous polynomial is a rather complicated object which is known to be irreducible in the universal setting over the integers (see for instance \cite[\S 4]{BuJo12}). The purpose of this section is to prove  that when the homogeneous polynomial is symmetric then its discriminant can be decomposed into the product of several resultants that are in principle easier to compute (see Theorem \ref{thm:discmaintheorem}). We will obtain this result by specialization of the two formulas given  in Theorem \ref{thm:maintheorem}. 

\medskip

Fix a positive integer $n\geq 2$. For any integer $p$ we will denote by $e_{p}(x_{1},\ldots,x_{n})$ the $p^{\mathrm{th}}$ elementary symmetric polynomial in the variables $x_{1},\ldots,x_{n}$. They satisfy to the equality
$$\sum_{p\geq 0} e_{p}(x )t^{p}=\prod_{i=1}^{n}(1+x_{i}t)$$
(observe that $e_{0}(x)=1$ and that $e_{p}(x)=0$ for all $p>n$). For any partition $\lambda=(\lambda_{1} \geq \cdots \geq \lambda_{k})$ we also define the polynomial
$$e_{\lambda}(x):=e_{\lambda_{1}}( x) e_{\lambda_{2}}( x) \cdots e_{\lambda_{k}}( x) \in \ZZ[x_{1},\ldots,x_{n}].$$
Given a positive integer $d$, it is well known that the set 
\begin{equation}\label{elambda}
 \{ e_{\lambda}(x) \ : \ \lambda=(\lambda_{1},\ldots,\lambda_{k}) \vdash d \textrm{ such that }  n \geq \lambda_{1} \geq \lambda_{2} \geq \cdots \geq \lambda_{k} \}
\end{equation}
is a basis (over $\ZZ$) of the homogeneous symmetric polynomials of degree $d$ in $n$ variables. In other words, any homogeneous symmetric polynomial of degree $d$ with coefficients in a commutative ring is obtained as specialization of the generic homogeneous symmetric polynomial of degree $d$
\begin{equation}\label{eq:F}
 F(x_{1},\ldots,x_{n}):=\sum_{\lambda \vdash d} c_{\lambda}e_{\lambda}(x) \in \ZZ[c_{\lambda} : \lambda \vdash d][x_{1},\ldots,x_{n}].
\end{equation}
We will denote by $\UU$ its universal ring of coefficients $\ZZ[c_{\lambda} : \lambda \vdash d]$. In addition, 
for all $i\in \{1,\ldots,n\}$, we will denote the partial derivatives of $F$ by
$$F^{\{i\}}(x_{1},\ldots,x_{n}):=\frac{\partial F}{\partial x_{i}} (x_{1},\ldots,x_{n}) \in \UU[x_{1},\ldots,x_{n}]_{d-1}.$$ 
Finally, we recall that the discriminant of $F$ is defined by the equality (see \S \ref{subsec:disc})
\begin{equation}\label{eq:discressym}
d^{a(n,d)}\Disc(F)=\Res\left(  F^{\{1\}},F^{\{2\}},\ldots,F^{\{n\}} \right) \in \UU
\end{equation}
and that it is homogeneous of degree $n(d-1)^{n-1}$ in $\UU$.

\begin{lem}
The partial derivatives $F^{\{1\}},F^{\{2\}},\ldots,F^{\{n\}}$ of the symmetric polynomial $F(x_{1},\ldots,x_{n})$ form  a $\mathfrak{S}_{n}$-equivariant polynomial system.
\end{lem}
\begin{proof} Since $F$ is a polynomial in the elementary symmetric polynomials, the chain rule formula for the derivation of composed functions shows that
there exist $\min\{d,n\}$ homogeneous symmetric polynomials $S_k(x_{1},\ldots,x_{n})$ such that for all $i=1,\ldots,n$ 
\begin{equation}\label{eq:diffF}
	F^{\{i\}}=\frac{\partial F}{\partial x_i} = \sum_{k=1}^{\min\{d,n\}} \frac{\partial e_{k}}{\partial x_{i}} S_{k}(x_{1},\dots,x_{n}).
\end{equation}
Moreover, for any pair of integers $i,j$ we have
\begin{equation}\label{eq:diffej}
 \frac{\partial e_{j}}{\partial x_i}=\sum_{r=0}^{j-1} (-1)^{r}x_{i}^{r}e_{j-1-r}.
\end{equation}
Therefore, we deduce that for any $\sigma \in \mathfrak{S}_{n}$, we have 
$\sigma	\left(F^{\{i\}} \right)= F^{\{\sigma(i)\}}$ as claimed.
\end{proof}

As a consequence of this lemma, Theorem \ref{thm:maintheorem} can be applied in order to decompose the resultant of the polynomials $F^{\{1\}},F^{\{2\}},\ldots,F^{\{n\}}$ and hence, by \eqref{eq:discressym}, to decompose the discriminant of the symmetric polynomial $F$. We take again the notation of \S \ref{subsec:divdiff} and  \S \ref{subsec:mainthm}.

\begin{thm}\label{thm:discmaintheorem} Assume that $n\geq 2$ and $d\geq 2$. With the above notation, the following equalities hold.

\noindent $\bullet$ If $d> n$ then 
\begin{equation*}
d^{a(n,d)}\Disc\left(F\right)=\prod_{\substack{\lambda\vdash n}}
\Res\left( F_{\lambda}^{\{1\}}, F_{\lambda}^{\{1,2\}}, \ldots, F_{\lambda}^{\{1,2,\ldots,l(\lambda)-1\}}, F_{\lambda}^{\{1,2,\ldots,l(\lambda)\}} 
\right)^{m_{\lambda}}.
\end{equation*}

\noindent $\bullet$ If $d\leq n$ then 
\begin{multline*}
d^{a(n,d)}\Disc\left( F \right)=  \left(F^{\{1,\ldots,d\}}\right)^{m_{0}}\prod_{ \substack{\lambda\vdash n \\ l(\lambda) <d}}
\Res\left( F_{\lambda}^{\{1\}}, F_{\lambda}^{\{1,2\}}, \ldots, F_{\lambda}^{\{1,2,\ldots,l(\lambda)-1\}}, F_{\lambda}^{\{1,2,\ldots,l(\lambda)\}} 
\right)^{m_{\lambda}}
\end{multline*}
where 
\begin{equation*}
 m_{0}:=n(d-1)^{n-1}-\sum_{ \substack{\lambda\vdash n \\ l(\lambda) < d}} m_{\lambda} 
\left(\sum_{j=1}^{l(\lambda)} \frac{(d-1)(d-2)\cdots(d-l(\lambda)) }{(d-j)} \right).
\end{equation*}
Moreover, if $F$ is given by \eqref{eq:F} then $F^{\{1,\ldots,d\}}=(-1)^{{d-1}}c_{(d)}$ so that 
\begin{multline*}
d^{a(n,d)}\Disc\left( F \right)=  (-1)^{\varepsilon}\left(c_{d}\right)^{m_{0}}\prod_{ \substack{\lambda\vdash n \\ l(\lambda) <d}}
\Res\left( F_{\lambda}^{\{1\}}, F_{\lambda}^{\{1,2\}}, \ldots, F_{\lambda}^{\{1,2,\ldots,l(\lambda)-1\}}, F_{\lambda}^{\{1,2,\ldots,l(\lambda)\}} 
\right)^{m_{\lambda}}
\end{multline*}
where $\varepsilon=n-1$ if $d=2$ and $\varepsilon=0$ if $d\geq 3$.
\end{thm}

\begin{proof} These formulas are obtained by specialization of the formulas given in Theorem \ref{thm:maintheorem} with the difference that the polynomials $F^{\{i\}}$, $i=1,\ldots,n$ are of degree $d-1$ in our setting (and not of degree $d$ as in Theorem \ref{thm:maintheorem}). Thus, the only thing we need to show is that 
\begin{equation}\label{eq:F1ddisc}
 F^{\{1,\ldots,d\}}=(-1)^{{d-1}}c_{(d)}
\end{equation}
under the assumption $n\geq d$, 
where $c_{(d)}$ is the coefficient of $F$ in the writing \eqref{eq:F} that corresponds to the partition $\lambda=(d)$. Indeed, by the above second equality for $m_{0}$ we see that $m_{0}$ is even if $d\geq 3$, whereas the first equality shows that $m_{0}=n-1 \mod 2$ if $d=2$.

To prove \eqref{eq:F1ddisc}, observe that \eqref{eq:diffF} and \eqref{eq:diffej} show that there exist symmetric homogeneous polynomials $S_{k}(x_{1},\ldots,x_{n})$, $k=1,\ldots,d$ of degree $d-k$ respectively, such that 
for all $i=1,\ldots,n$
\begin{equation}\label{eq:combFidisc}
F^{\{i\}}= \sum_{k=1}^{d} \frac{\partial e_{k}}{\partial x_{i}} S_{k}(x_{1},\dots,x_{n})=
\sum_{k=1}^{d}\sum_{r=0}^{k-1}(-1)^{r}x_{i}^{r} e_{k-1-r}S_{k}=\sum_{r=0}^{d-1}x_{i}^{r} \left( \sum_{k=r+1}^{d}(-1)^{r} e_{k-1-r}S_{k} \right). 
\end{equation}
Now, by the defining equality of divided differences given in Lemma \ref{lem:divdiff}, we have 
\begin{equation*}
V(x_{1},x_{2},\ldots,x_{d})\cdot F^{\{ 1,\ldots,d\}}=
\left|
\begin{array}{ccccc}
 1 & x_{{1}} & \cdots & x_{{1}}^{{d-2}} & F^{\{1\}}\\
 1 & x_{{2}} & \cdots & x_{2}^{{d-2}} & F^{\{2\}}\\
 \vdots & \vdots & & \vdots & \vdots \\
  1 & x_{{d}} & \cdots & x_{{d}}^{{d-2}} & F^{\{n\}} \\
\end{array}
\right|. 
\end{equation*}
Therefore, using \eqref{eq:combFidisc} one can reduce, by elementary operations on columns, the last column of the above determinant to terms corresponding to the indexes $k=d,r=d-1$, that is to say 
$$
\left|
\begin{array}{ccccc}
 1 & x_{{1}} & \cdots & x_{{1}}^{{d-2}} & F^{\{1\}}\\
 1 & x_{{2}} & \cdots & x_{2}^{{d-2}} & F^{\{2\}}\\
 \vdots & \vdots & & \vdots & \vdots \\
  1 & x_{{d}} & \cdots & x_{{d}}^{{d-2}} & F^{\{n\}} \\
\end{array}
\right|=
\left|
\begin{array}{ccccc}
 1 & x_{{1}} & \cdots & x_{{1}}^{{d-2}} & (-1)^{d-1}x_{1}^{d-1}S_{d}\\
 1 & x_{{2}} & \cdots & x_{2}^{{d-2}} & (-1)^{d-1}x_{2}^{d-1}S_{d}\\
 \vdots & \vdots & & \vdots & \vdots \\
  1 & x_{{d}} & \cdots & x_{{d}}^{{d-2}} &  (-1)^{d-1}x_{1}^{d-1}S_{d}
\end{array}
\right|.$$
It follows that $F^{\{1,\ldots,d\}}=(-1)^{{d-1}}S_{d}$.  Finally, from the definition \eqref{eq:F} of $F$, we have $S_{d}=c_{(d)}$ and the proof is completed.
\end{proof}

We emphasize that the formulas given in this theorem are universal with respect to the coefficients of $F$ and are independent of the choice of basis that is used to represent $F$ (for the sake of generality, we have chosen the basis  \eqref{elambda} as an illustration).  We also mention that formulas similar to the ones given in \S \ref{subset:averaging} can also be written explicitly for the discriminant of $F$  (this is actually the point of view that has been used in \cite{PeSh09}). 

\medskip
 
Hereafter, we give two examples corresponding to low degree polynomials, namely the cases  $d=2$ and $d=3$. In these two cases the number of variables $n$ is large compared to $d$ and the formulas given in Theorem \ref{thm:discmaintheorem} are hence computationally very interesting since a resultant computation in $n$ variables is replaced by several resultant computations in at most $d$ variables. 

\medskip

\paragraph{\bf Case $\mathbf{n\geq d=2}$} The generic homogeneous polynomial of degree 2 can be written as 
$$F=c_{(2)}e_{2}+c_{(1,1)}e_{1}^{2}.$$
Its derivatives are 
$$F^{\{i\}}=c_{(2)}\frac{\partial e_{2}}{\partial x_{1}}+2c_{(1,1)}e_{1}\frac{\partial e_{1}}{\partial x_{1}}=c_{(2)}(e_{1}-x_{1})+2c_{(1,1)}e_{1}$$
and hence we deduce that
$$\Res\left(F_{(2)}^{\{1\}}\right)=(n-1)c_{(2)}+2nc_{(1,1)}.$$
Observe that this polynomial is not irreducible over $\ZZ[c_{(2)},c_{(1,1)}]$ if $n$ is odd since it is divisible by 2. It is also not hard to check that $m_{(2)}=1$ and $m_{0}=n-1$ here. Finally, since $a(n,2)=0$ if $n$ is even and $a(n,2)=1$ if $n$ is odd, we get 
$$ \Disc(F)=
\begin{cases}
 -c_{(2)}^{n-1}\left((n-1)c_{(2)}+2nc_{(1,1)} \right) & \textrm{ if } n \textrm{ is even},\\
c_{(2)}^{n-1}\left( \frac{n-1}{2}c_{(2)}+nc_{(1,1)} \right) &  \textrm{ if } n \textrm{ is odd }.
\end{cases}
$$

\paragraph{\bf Case $\mathbf{n\geq d=3}$} Consider the  generic homogeneous polynomial of degree 3
$$F=c_{(3)}e_{3}+c_{(2,1)}e_{2}e_{1}+c_{(1,1,1)}e_{1}^{3}.$$
The formula given in Theorem \ref{thm:discmaintheorem} shows that
$$ 3^{\frac{2^{n}-(-1)^{n}}{3}}\Disc(F)=c_{(3)}^{m_{0}}\Res\left(F_{(n)}^{\{1\}}\right) \prod_{k=1}^{\lfloor \frac{n}{2} \rfloor} 
\Res\left(F_{(n-k,k)}^{\{1\}},F_{(n-k,k)}^{\{1,2\}}\right)^{m_{(n-k,k)}}$$
where all the factors can be described explicitly. To begin with, from \eqref{eq:diffF} and \eqref{eq:diffej} we get that for all $i=1,\ldots,n$
$$F^{\{i\}}=c_{(3)}\left( e_{2}-x_{i}e_{1}+x_{i}^{2}   \right) +c_{(2,1)}\left( e_{2}+e_{1}(e_{1}-x_{i}) \right)+3c_{(1,1,1)}e_{1}^{2}.$$
It follows immediately that
$$\Res\left(F_{(n)}^{\{1\}}\right)=\binom{n-1}{2}c_{(3)}+3\binom{n}{2}c_{(2,1)}+3n^{2}c_{(1,1,1)}.$$
Now, let $(n-k,k)$ be a partition of length 2 of $n$. A straightforward computation shows that for any pair of distinct integers $i,j$ we have
$$F^{\{i,j\}}=c_{(3)}\left( x_{i}+x_{j}-e_{1} \right)-c_{(2,1)}e_{1}$$
and we deduce, by means of a single (Sylvester) resultant computation that
\begin{multline*}
 \Res\left(F_{(n-k,k)}^{\{1\}},F_{(n-k,k)}^{\{1,2\}}\right)= c_{(3)}^{2}\left(  \binom{n-1}{2}c_{(3)}+3\binom{n}{2}c_{(2,1)}+3n^{2}c_{(1,1,1)}  \right)  \\
-\frac{1}{2}k(n-k)\left( 
 \left( n-2 \right) c_{{(3)}}^{3}+ \left( 24c_{{(1,1,1)}}+3nc_{{(2,1)}}
 \right) c_{{(3)}}^{2}+ \left( 3n-6 \right) c_{{(2,1)}}^{2}c_{{(3)}}+nc_{(2,1)}^{3}
    \right).
\end{multline*}
The multiplicity $m_{(n-k,k)}$ are equal to the binomial $\binom{n}{k}$ for all $k=1,\ldots,\lfloor \frac{n}{2}\rfloor$ except if $n$ is even and $k=\frac{n}{2}$ in which case $m_{(\frac{1}{2},\frac{1}{2})}=\frac{1}{2}\binom{n}{\frac{n}{2}}$. Finally, it remains to determine the integer $m_{0}$. We have
$$m_{0}=n2^{n-1}-m_{(n)}-3\sum_{\substack{\lambda\vdash n \\ l(\lambda)=2}} m_{\lambda}=n2^{n-1}-1-3\sum_{k=1}^{\lfloor \frac{n}{2}\rfloor} m_{(n-k,k)}.$$
But since
$$2\sum_{k=1}^{\lfloor \frac{n}{2}\rfloor} m_{(n-k,k)}=\sum_{k=1}^{n-1}\binom{n}{k}=2^{n}-2=2(2^{n-1}-1),$$
we finally deduce that $$m_{0}=(n-3)2^{n-1}+2.$$ To illustrate this general formula, we detail the two particular cases $n=3$ and $n=4$. If $n=3$, we obtain
$$  \Disc(F)={c_{{(3)}}}^{2} \left( c_{{(3)}}+9\,c_{{(2,1)}}+27\,c_{{(1,1,1)}} \right) 
 \left( -{c_{{(2,1)}}}^{2}c_{{(3)}}-{c_{{(2,1)}}}^{3}+c_{{(1,1,1)}}{c_{{(3)}}}^{2}
 \right) ^{3}$$
 where
 $$\Res\left(F_{(3)}^{\{1\}}\right)=\left( c_{{(3)}}+9\,c_{{(2,1)}}+27\,c_{{(1,1,1)}} \right), \ m_{{(3)}}=1$$
 and
$$\Res\left(F_{(2,1)}^{\{1\}},F_{(2,1)}^{\{1,2\}}\right)=3\left( -{c_{{(2,1)}}}^{2}c_{{(3)}}-{c_{{(2,1)}}}^{3}+c_{{(1,1,1)}}{c_{{(3)}}}^{2}
 \right), \ m_{(2,1)}=3.$$
If $n=4$ we get
\begin{multline}\label{eq:discn=4}
 \Disc(F)=-{c_{{(3)}}}^{10} \left( c_{{(3)}}+2\,c_{{(2,1)}} \right) ^{9} \left( 6
\,c_{{(2,1)}}+16\,c_{{(1,1,1)}}+c_{{(3)}} \right) \times \\ \left( 4\,c_{{(1,1,1)}}{c_{{(3)}}}
^{2}-3\,{c_{{(2,1)}}}^{2}c_{{(3)}}-2\,{c_{{(2,1)}}}^{3} \right) ^{4}
\end{multline}
where
$$\Res\left(F_{(4)}^{\{1\}}\right)=3\left( 6
\,c_{{(2,1)}}+16\,c_{{(1,1,1)}}+c_{{(3)}} \right),  \ m_{(4)}=1,$$
$$\Res\left(F_{(3,1)}^{\{1\}},F_{(3,1)}^{\{1,2\}}\right)= 3 \left( 4\,c_{{(1,1,1)}}{c_{{(3)}}}
^{2}-3\,{c_{{(2,1)}}}^{2}c_{{(3)}}-2\,{c_{{(2,1)}}}^{3} \right), \ m_{(3,1)}=4$$
and
\begin{equation}\label{eq:notirredres}
 \Res\left(F_{(2,2)}^{\{1\}},F_{(2,2)}^{\{1,2\}}\right)= -\left( c_{{(3)}}+2\,c_{{(2,1)}} \right) ^{3}, \  m_{(2,2)}=3.
\end{equation}
For instance, for the particular example of the Clebsch surface which is given by the equation
$$h(x_{1},x_{2},x_{3},x_{4})=x_{1}^{3}+x_{2}^{3}+x_{3}^{3}+x_{4}^{3}-(x_{1}+x_{2}+x_{3}+x_{4})^{3}=3e_{3}-3e_{2}e_{1}=0,$$
we recover the known fact that $h/3$ defines a smooth cubic in every characteristic except 5 (see \cite[\S 5.4]{Saito}) since \eqref{eq:discn=4} shows that 
$$\Disc(h/3)=\Disc(e_{3}-e_{2}e_{1})=-(-1)^{9}(-6+1)(-3+2)^{4}=-5.$$

\begin{rem} Contrary to what was expected in \cite{PeSh09}, the resultant factors appearing in Theroem \ref{thm:discmaintheorem}  are not always irreducible (see e.g.~\eqref{eq:notirredres}). However, we ignore if these resultant factors are geometrically irreducible (i.e.~are irreducible polynomials up to a certain power) when the ground ring is assumed to be field, but this was the case in all the experiments that we have done. As an illustration, we notice that the factor \eqref{eq:m=n/2} appearing in Example \ref{ex:d=2} is not geometrically irreducible, but it becomes geometrically irreducible (over a field) when specialized to get the discriminant formula in the case $n\geq d=3$. Indeed, comparing the notation in these two examples we get
$d=-b=c_{(3)}+c_{(2,1)}.$
\end{rem}

\paragraph{\small{\textsc{Acknowledgments}.}}

The authors are grateful to Evelyne Hubert for useful discussions on equivariant polynomial systems. 
The second author's research has
received funding from the European Union (European Social Fund) 
and Greek national funds through the Operational Program 
``Education and Lifelong Learning" of the National Strategic 
Reference Framework, Research Funding Program ``ARISTEIA",
Project ESPRESSO: Exploiting Structure in Polynomial Equation and System
Solving with Applications in Geometric and Game Modeling.
She also acknowledges the Galaad project team at INRIA Sophia-Antipolis
that made possible her visit to INRIA.


\end{document}